\documentclass[11pt]{article}
\usepackage{amsfonts}
\usepackage{mathrsfs}
\usepackage{amssymb}
\usepackage{amsmath}
\usepackage{amsthm}
\usepackage{latexsym}
\usepackage{graphicx}
\usepackage{color}
\usepackage{hyperref}
\usepackage{ulem}
\usepackage{fancyhdr}
\usepackage{lastpage}
 \numberwithin{equation}{section}

\newtheorem{Theorem}{Theorem}[section]
\newtheorem{Lemma}{Lemma}[section]

\newtheorem{Proposition}{Proposition}[section]
\newtheorem{definition}{Definition}[section]
\newtheorem{Remark}{Remark}[section]
\allowdisplaybreaks

\begin{document}

\begin{center}


        \vspace{1cm}
        \LARGE{\textbf{On the Vanishing Viscosity and Magnetic Diffusion Limit for Incompressible Magnetohydrodynamics}}
        \vspace{1cm}

        \large
        {\bf \v{S}\'{a}rka Ne\v{c}asov\'{a}}\\
        Institute of Mathematics of the Academy of Sciences of the Czech Republic,\\
        Zitn\'{a} 25, 11567, Praha 1, Czech Republic\\
        E - mail: matus@math.cas.cz\\
        \vspace{1cm}
       {\bf Tong Tang}\footnote{Corresponding author.}\\
        Department of Mathematics, College of Sciences\\
        Yangzhou University, Yangzhou 225000, China\\
        E - mail: tt0507010156@126.com\\
        \vspace{1cm}
        {\bf Yuanjing Tang}\\
        Department of Mathematics, College of Sciences\\
        Yangzhou University, Yangzhou 225000, China\\
        E - mail: Tyj18118851855@163.com\\
        \vspace{1cm}
        {\bf Lu Zhu}\\
        School of Mathematics and Laboratory of Mathematical
        Modeling and Intelligent Computing for Water Systems,
        Hohai University, Nanjing, P.R. China\\
        E - mail: zhulu@hhu.edu.cn\\
        \vspace{0cm}

        \rule{\textwidth}{1pt}
    \end{center}

	\begin{abstract}
	The article focuses on the inviscid and non-resistive limit of the incompressible magnetohydrodynamics (MHD) equations in a bounded domain. We derive sufficient conditions for the limit of Leray-Hopf solutions of the MHD system, which converges to the weak solutions of the ideal MHD system when $\mu,\nu\rightarrow0$, under the assumptions of uniform interior regularity and uniform equi-continuity at the boundary.

Therefore, we establish an analogue of the Kato-type criterion to show that the energy dissipation rate is independent of the coefficient of viscosity and resistivity.
Moreover, to analyze the different decay rates for the viscosity and resistivity, we obtain the limit of infinite and vanishing magnetic Prandtl number ($Pm=\frac{\mu}{\nu}$) for the incompressible MHD system. To the best of the authors' knowledge, this is the first rigorous result from a mathematical viewpoint on the limit of the magnetic Prandtl number.
	\end{abstract}


	\vspace{0.2 cm}
	

\section{Introduction}
\hspace{1.5em}The incompressible magnetohydrodynamics (MHD) system is coupled by the Navier-Stokes equations and Maxwell equations, which describe the dynamics of electrically conducting fluids, such as astrophysical and laboratory plasmas or liquid metals. In this paper, we introduce two important MHD systems.
The first model describes incompressible fluids coupled with a magnetic field, which includes resistive diffusion:
\begin{equation}
\label{1.1}\begin{cases}
\partial_t \mathbf{u} + (\mathbf{u} \cdot \nabla )\mathbf{u} - \mu \Delta \mathbf{u} + \nabla P = (\mathbf{B} \cdot \nabla) \mathbf{B},\\[4pt]
\partial_t \mathbf{B} + (\mathbf{u} \cdot \nabla )\mathbf{B} - \nu \Delta \mathbf{B} = (\mathbf{B} \cdot \nabla) \mathbf{u},\\[4pt]
\nabla \cdot \mathbf{u} = \nabla \cdot \mathbf{B} = 0, \\[4pt]
(\mathbf{u}, \mathbf{B})|_{t=0} = (\mathbf{u}_0, \mathbf{B}_0).
\tag{1.1}
\end{cases}
\end{equation}
The system \eqref{1.1} is supplemented with the following boundary conditions:
\[
\mathbf{u}|_{\partial\Omega} = \mathbf{0}, \quad
\mathbf{B} \cdot \mathbf{n}|_{\partial\Omega} = 0, \quad
(\nabla \times \mathbf{B}) \times \mathbf{n}|_{\partial\Omega} = \mathbf{0},
\]
where \(\Omega \subset \mathbb{R}^3\) is a bounded domain with smooth boundary \(\partial\Omega\), and \(\mathbf{n}\) denotes the unit outward normal vector. Here, \(\mathbf{u} = \mathbf{u}(t, x)\) and \(\mathbf{B} = \mathbf{B}(t, x)\) are the velocity and magnetic fields, respectively, with \(t \in \mathbb{R}^+\) and \(x \in \Omega\). The scalar pressure is denoted by \(P = P(t, x)\), while \(\mu > 0\) and \(\nu > 0\) represent the fluid viscosity and magnetic resistivity, respectively. The dissipative terms \(\mu \Delta \mathbf{u}\) and \(\nu \Delta \mathbf{B}\) account for viscous damping of fluid motion and resistive diffusion of the magnetic field.

Next, we turn to the ideal MHD system (inviscid)  as $\mu, \nu \to 0$ in \eqref{1.1}:
\[
\begin{cases}
\partial_t \mathbf{u} + (\mathbf{u} \cdot \nabla )\mathbf{u} + \nabla P = (\mathbf{B} \cdot \nabla )\mathbf{B}, \\[4pt]
\partial_t \mathbf{B} + (\mathbf{u} \cdot \nabla )\mathbf{B} = (\mathbf{B} \cdot \nabla )\mathbf{u}, \\[4pt]
\nabla \cdot \mathbf{u} = \nabla \cdot \mathbf{B} = 0, \\[4pt]
(\mathbf{u}, \mathbf{B})|_{t=0} = (\mathbf{u}_0, \mathbf{B}_0),
\tag{1.2}\label{1.2}
\end{cases}
\]
and we supplement the corresponding boundary conditions as:
\[
\mathbf{u} \cdot \mathbf{n}|_{\partial\Omega} = 0, \quad
\mathbf{B} \cdot \mathbf{n}|_{\partial\Omega} = 0.
\]

Turbulence is one of the most challenging problems in classical physics. Since the pioneering K41 theory, people realized that turbulence is intimately connected to anomalous energy dissipation. Tremendous experimental observations and numerical simulations have proved Kolmogorov's prediction that the energy dissipation rate is independent of the viscosity at high Reynolds numbers. Therefore, it is a fundamental open problem to prove the Navier-Stokes equations converge to inviscid Euler equations at the infinite Reynolds number limit, that is, the so-called zero-dissipation limit or vanishing viscosity limit. Regarding the weak solutions, it is impossible to pass the limit due to the lack of compactness. A seminal breakthrough was achieved by Kato \cite{kato84}, who established a criterion for the inviscid limit based on the vanishing of energy dissipation within a boundary layer. Readers can refer to \cite{brkic24,maekawa14,mazz26,sam98} for details and background.  Compared with traditional statistical theory, Onsager proposed that the energy dissipation is governed by the regularity of the velocity instead of a random process, which is the celebrated Onsager conjecture. It states that the weak solutions of incompressible fluid conserve kinetic energy or experience anomalous dissipation with the threshold regularity $C^\alpha$.  The sufficient regularity condition guaranteeing energy conservation was analyzed from the perspective of rigorous mathematical proof by Eyink \cite{eyink94} at Fourier space, and extended to the Besov space by Constantin, E and Titi \cite{const94}. After that, it was extended to various cases of fluid, see \cite{bardos19,ca,feireisl17,kang07}. 
Under the assumption of interior $C^\alpha$ regularity and behavior near the boundary, Bardos, Titi, and Wiedemann \cite{bardos19} showed the Leray-Hopf solutions of the Navier-Stokes system converge to the weak solutions of Euler at the vanishing viscosity limit. Recently, Chen, Liang, and Wang \cite{chen22} showed the convergence from weak solutions of Navier-Stokes equations to Euler equations by constructing a new boundary layer foliation.



It is well known that turbulence in electrically conducting media has not yet been fully understood due to the presence of magnetic fields. In this vein, it is necessary to consider the MHD turbulence. Recently, Faraco and Lindberg \cite{faraco20} showed that magnetic helicity is conserved in simply connected domains in the limit ($\mu,\nu\rightarrow0$) under $L^2_tL^2_x$ integrability, while their weak limits of velocity and magnetic field do not satisfy the ideal MHD system. We should emphasize that the vanishing dissipation limit for the MHD system is not a straightforward corollary of the Navier-Stokes results as there are some essential difficulties and differences.  Precisely, the transition involves strongly coupled boundary layers that arose from the different boundary conditions: the no-slip condition for the velocity field versus the perfectly conducting wall condition for the magnetic field. On the other hand, if we consider the two different dissipation coefficients $\mu=\varepsilon^p$ and $\nu=\varepsilon^q$, they generate boundary layers of distinct thicknesses that are mutually coupled. Last but not least, the nonlinear interactions driven by the Lorentz force $(\mathbf{B} \cdot \nabla)\mathbf{B}$ and the magnetic induction $(\mathbf{u} \cdot \nabla)\mathbf{B}$ bring severe analytical difficulties.  For more results and physical explanation, see \cite{ca,faraco24,kang07,kozono89,lin15}.

\textbf{In this paper, we develop a new sufficient Kato-type criterion and show that the total energy dissipation rate is independent of the coefficient of viscosity and resistivity (see Theorem 1.1). Moreover, in contrast to \cite{faraco20}, we prove the weak limits of the Leray-Hopf solutions for the MHD system, which satisfy the ideal MHD equations and our methodology relies on two crucial elements: boundary layer foliation and variables Els\"asser.}

Firstly, motivated by \cite{chen22}, we construct a refined \textit{boundary layer foliation} for the MHD system, which is essential for handling the different scales of the dissipation coefficients. In order to decouple convective and magnetic stretching terms, we reformulate nonlinear terms by the \textit{Els\"asser variables}. More precisely, Els\"asser variables represent the symmetric combinations of the fluid velocity and magnetic fields. Then, applying a spatial localization and averaging procedure to these variables, we can analyze the combination of nonlinear fluxes at distinct foliations. Based on this strategy, we construct corresponding symmetric nonlinear functions (\eqref{3.58},\eqref{3.59},\eqref{3.63},\eqref{3.64}) for each foliation.   



Before stating our main result, we recall the definition of weak solutions for the dissipative MHD system. It is Sermange and Temam \cite{ser83} that established the global existence of Leray-Hopf weak solutions.

\begin{definition}
For any fixed \(\varepsilon > 0\) and initial data \((\mathbf{u}^\varepsilon_0, \mathbf{B}^\varepsilon_0) \in L^2_\sigma(\Omega)\)
\footnote{ Here and afterward, we define
$L^2_\sigma = \{ \mathbf{v} \in L^2(\Omega) : \text{div}\,\mathbf{v} = 0, \left.\mathbf{v}\cdot\mathbf{n}\right|_{\partial\Omega} = 0 \}$,
$W^{1,2}_{\sigma} = \{ \mathbf{v} \in W^{1,2}(\Omega) : \text{div}\,\mathbf{v} = 0, \left.\mathbf{v}\cdot\mathbf{n}\right|_{\partial\Omega} = 0  \}$, $W^{1,2}_{0,\sigma} = \{ \mathbf{v} \in W_0^{1,2}(\Omega) : \text{div}\,\mathbf{v} = 0 \}$. }, a pair \((\mathbf{u^\varepsilon}, \mathbf{B^\varepsilon})\) is called a \emph{Leray-Hopf weak solution} to the dissipative MHD system (1.1) if
\begin{itemize}
     \item[(i)] \(\mathbf{u^\varepsilon}\in L^\infty((0,T);L^2_{\sigma}(\Omega))\cap L^2((0,T);W_{0,\sigma}^{1,2}(\Omega)),  \mathbf{B^\varepsilon}\in L^\infty((0,T);L^2_{\sigma}(\Omega))\cap L^2((0,T);W_{\sigma}^{1,2}(\Omega))\);
    \item[(ii)] \((\mathbf{u^\varepsilon}, \mathbf{B^\varepsilon})\) solves (1.1) in the sense of distributions;
    \item[(iii)] the global energy inequality holds for almost every \(t \in (0, T]\):
    \begin{equation}
     \frac{1}{2} \int_{\Omega} \left( |\mathbf{u}^\varepsilon(t)|^2 + |\mathbf{B}^\varepsilon(t)|^2 \right) dx + \int_0^t \int_{\Omega} \left( \mu |\nabla \mathbf{u}^\varepsilon|^2 + \nu |\nabla \mathbf{B}^\varepsilon|^2 \right) dx\,ds \le \frac{1}{2} \int_{\Omega} \left( |\mathbf{u}_0^\varepsilon|^2 + |\mathbf{B}_0^\varepsilon|^2 \right) dx.
     \tag{1.3}\label{1.3}
    \end{equation}
\end{itemize}
\end{definition}

As pointed out in \cite{bra}, the magnetic Prandtl number $P_{m}=\frac{\mu}{\nu}$ is a significant parameter to study the MHD turbulence. Here we assume
\begin{align*}
\mu = \varepsilon^p \quad \text{and} \quad \nu = \varepsilon^q, \quad \text{for some } p, q > 0,
\end{align*}
so that we could obtain $P_m$ goes to infinity or vanish when $\epsilon\rightarrow0$. Moreover, in order to control the boundary layer effects from the velocity and magnetic fields, we define the \textit{unified boundary layer thickness} \(\delta(\varepsilon)\) by
\[
\delta(\varepsilon) := \min\{\varepsilon^p, \varepsilon^q\}.
\]

Throughout this paper, for any \(h>0\), we denote the interior domain by \(\Omega^h := \{x \in \Omega : \operatorname{dist}(x, \partial\Omega) > h\}\) and the boundary layer region by \(\Gamma_h := \Omega \setminus \overline{\Omega}^h\).

We are now in a position to state our main theorem.
\begin{Theorem}
Let \(\Omega \subset \mathbb{R}^3\) be a bounded domain with \(C^2\) boundary and let \(T \in (0, \infty)\). For each \(\varepsilon > 0\), let \((\mathbf{u}^\varepsilon, \mathbf{B}^\varepsilon)\) be a Leray-Hopf weak solution of the dissipative MHD system with initial data \((\mathbf{u}_0^\varepsilon, \mathbf{B}_0^\varepsilon)\), and suppose that
\[
(\mathbf{u}_0^\varepsilon, \mathbf{B}_0^\varepsilon) \to (\mathbf{u}_0, \mathbf{B}_0) \quad \text{in } L^2(\Omega) \text{ as } \varepsilon \to 0+.
\tag{1.4}\label{1.4}
\]
Assume, in addition, that the sequence satisfies the uniform regularity bounds
\[
(\mathbf{u}^\varepsilon, \mathbf{B}^\varepsilon) \text{ is uniformly bounded in } L^3(0, T; B_{3}^{\alpha_1,\infty}(\Omega^\delta) \times B_{3}^{\alpha_2,\infty}(\Omega^\delta)),
\tag{1.5}\label{1.5}
\]
where the spatial H\"older exponents \(\alpha_1\) and \(\alpha_2\) satisfy the following regime-dependent thresholds:
\begin{equation}
\tag{1.6}\label{1.6}
\begin{array}{l}
\text{(i)}\ \alpha_1, \alpha_2 \in (\frac{5}{6} - \frac{q}{2p}, \frac{5}{6}),\ \text{if}\ p \ge q;\\
\text{(ii)}\ \alpha_1, \alpha_2 \in (\frac{5}{6} - \frac{p}{2q}, \frac{5}{6}),\  \text{if}\ p < q.
\end{array}
\end{equation}
Furthermore, suppose that on the boundary layer region \(\Gamma_{\delta(\varepsilon)}\), we have
\begin{equation}
\begin{cases}
\mathbf{u}^\varepsilon, \mathbf{B}^\varepsilon \text{ are uniformly bounded in } L^4(0, T; L^\infty(\Gamma_{4\delta})), \\[4pt]
P^\varepsilon \text{ is uniformly bounded in } L^2(0, T; L^\infty(\Gamma_{4\delta})).
\end{cases}\label{1.7}
\tag{1.7}
\end{equation}
If the total dissipation within the boundary layer vanishes in the limit,
\[
\lim_{\varepsilon \to 0+}
\int_0^T \int_{\Gamma_{4\delta}} \bigl( \varepsilon^p|\nabla \mathbf{u}^\varepsilon|^2 +\varepsilon^q|\nabla \mathbf{B}^\varepsilon|^2 \bigr) dxdt = 0,
\tag{1.8}\label{1.8}
\]
then the global viscous and resistive dissipation vanishes, i.e.,
\[
\lim_{\varepsilon \to 0+} \int_0^T \int_{\Omega} \bigl( \varepsilon^p|\nabla \mathbf{u}^\varepsilon|^2 + \varepsilon^q|\nabla \mathbf{B}^\varepsilon|^2 \bigr) dxdt = 0.
\tag{1.9}\label{1.9}
\]
Moreover, up to a subsequence, \((\mathbf{u}^\varepsilon, \mathbf{B}^\varepsilon)\) converges strongly in \(L^3(0, T; L^3(\Omega))\) to a weak solution of the ideal MHD equations.
\end{Theorem}

\begin{Remark}
 Faraco and Lindberg \cite{faraco20} considered the simultaneous vanishing viscosity and magnetic diffusion limit as $\mu,\nu\rightarrow0$, where we could not deduce the limit for magnetic Prandtl number. In comparison with \cite{faraco20}, we choose $\mu=\varepsilon^p$ and $\nu=\varepsilon^q$ with different decaying rate, so that we could obtain $P_m$ goes to infinity or vanish when $\epsilon\rightarrow0$. To the best knowledge of the authors, this is the first rigorous result from mathematical viewpoint on the limit of magnetic Prandtl number.
\end{Remark}
The remainder of the paper is organized as follows. Section 2 collects the necessary mathematical preliminaries, including commutator estimates, pressure estimates, and the boundary layer foliation. Section 3 is devoted to the proof of Theorem 1.1. In this section, we employ Els\"asser variables to decouple the system and analyze the vanishing of the nonlinear energy flux. Finally, we provide a brief conclusion in Section 4.
\section{Preliminaries}
\subsection{Commutator and pressure estimates}
\hspace{1.5em}Recall the standard mollification of a function \(f\):
\begin{equation}\label{2.1}
\overline{f}_\tau(x) := \int_{B_\tau(0)} f(x - y)\eta_\tau(y)\,dy, \qquad \forall \, x \in \Omega^\tau,
\end{equation}
where \(\eta_\tau\) is the standard mollifier of width \(\tau\).
We first recall the definition of the Besov space \(B^{\alpha,\infty}_p(\Omega)\) for \(p \in [1, \infty]\) and \(\alpha \in (0, 1)\) with the norm:
\begin{equation}\label{2.2}
\|f\|_{B^{\alpha,\infty}_p(\Omega)} := \|f\|_{L^p(\Omega)} + [f]_{B^{\alpha,\infty}_p(\Omega)},
\end{equation}
where the H\"older-type semi-norm is defined as
\begin{equation}\label{2.3}
[f]_{B^{\alpha,\infty}_p(\Omega)} := \sup_{|y|>0} \frac{\|f(\cdot + y) - f(\cdot)\|_{L^p(\Omega \cap (\Omega - y))}}{|y|^\alpha}.
\end{equation}
And for any \(f \in B_r^{\alpha,\infty}(\Omega)\) with \(r \in [1,\infty]\), we have
\begin{equation}\label{2.4}
\|\nabla \overline{f}_\tau\|_{L^r(\Omega^\tau)} \leq C \tau^{\alpha-1} \|f\|_{B_r^{\alpha,\infty}(\Omega)},
\end{equation}
and
\begin{equation}\label{2.5}
\|\overline{f}_\tau - f\|_{L^r(\Omega^\tau)} \leq C\tau^{\alpha} \|f\|_{B_r^{\alpha,\infty}(\Omega)}.
\end{equation}

\begin{Lemma}[\cite{chen22}] \label{lem:commutator}
Let \( f \in B_{r_1}^{\alpha,\infty}(\Omega) \) and \( g \in B_{r_2}^{\alpha,\infty}(\Omega) \) with exponents satisfying \( 1 \leq r, r_1, r_2 < \infty \) and
\begin{equation}\label{2.6}
\frac{1}{r_1} + \frac{1}{r_2} = \frac{1}{r}.
\end{equation}
Then we have the following estimate:
\begin{equation}\label{2.7}
\| (f \otimes g)_{\tau} - \bar{f}_{\tau} \otimes \bar{g}_{\tau} \|_{L^r(\Omega^{\tau})} \leq C \tau^{2\alpha} \| f \|_{B_{r_1}^{\alpha,\infty}(\Omega)} \| g \|_{B_{r_2}^{\alpha,\infty}(\Omega)},
\end{equation}
where the constant \( C \) is independent of \( \tau \).
\end{Lemma}

To control the pressure gradient error within the boundary layer, we recall the following a priori estimate for the pressure:

\begin{Lemma}[\cite{esc17,ne}]\label{l2}
Let \(p \in (1, \infty)\). Suppose that \(\mathbf{u}^\varepsilon, \mathbf{B}^\varepsilon \in L^{2p}(\Omega)\) and \(P^\varepsilon|_{\partial\Omega} \in L^p(\partial\Omega)\). It follows that \(P^\varepsilon \in L^p(\Omega)\), and the following estimate holds:
\begin{equation}\label{2.8}
\|P^\varepsilon\|_{L^p(\Omega)} \leq C \big( \|P^\varepsilon|_{\partial\Omega}\|_{L^p(\partial\Omega)} + \|\mathbf{u}^\varepsilon\|_{L^{2p}(\Omega)}^2 + \|\mathbf{B}^\varepsilon\|_{L^{2p}(\Omega)}^2 \big).
\end{equation}
\end{Lemma}

Furthermore, to handle the nonlinear terms near the physical boundary, we utilize a standard Hardy-type inequality as the following:

\begin{Lemma}[\cite{kufner77}]
Let \(p \in [1, \infty)\) and \(f \in W_0^{1,p}(\Omega)\) such that the following inequality holds:
\begin{equation}\label{2.9}
\left\| \frac{f(x)}{\operatorname{dist}(x, \partial\Omega)} \right\|_{L^p(\Omega)} \leq C \|\nabla f\|_{L^p(\Omega)},
\end{equation}
where the constant \(C\) depends on \(p\) and \(\Omega\).
\end{Lemma}
\subsection{Boundary layer foliation}
\label{subsec:viscous_estimation}
\hspace{1.5em}   We introduce the boundary layer foliation, which plays an important role in the proof. Inspired by the work \cite{chen22}, we introduce the boundary layer foliation with an appropriate modification with respect of two boundary layers.  Let $\sigma = \min\{\alpha_1, \alpha_2\}$ denote the effective regularity index. For $\sigma \in\left(\dfrac{5}{6}-\min\left\{\dfrac{q}{2p},\dfrac{p}{2q}\right\}, \dfrac{5}{6}\right)$, we define the sequence
\begin{equation}\label{2.10}
\beta_0^* = 0 \quad \text{and} \quad \beta_n^* =
\begin{cases}
\displaystyle \frac{1}{2(1-\sigma)} \left( \frac{q}{p} + \frac{1}{3}\beta_{n-1}^* \right), & p \ge q, \\[10pt]
\displaystyle \frac{1}{2(1-\sigma)} \left( \frac{p}{q} + \frac{1}{3}\beta_{n-1}^* \right), & p < q.
\end{cases}
\quad , n=1,2,\dots,
\end{equation}
Clearly, \(\{\beta_n^*\}\) is bounded and strictly increasing, and
\begin{equation}\label{2.11}
\lim_{n \to \infty} \beta_n^* =
\begin{cases}
\displaystyle \frac{3 }{5-6\sigma}\cdot\frac{q}{p}>1, & p \ge q, \\[10pt]
\displaystyle \frac{3 }{5-6\sigma}\cdot\frac{p}{q}>1, & p < q.
\end{cases}
\end{equation}
Then we choose the integral number $N$ as the following:
\begin{align}
\label{2.12}&3q > p \ge q: \quad
N:=
\begin{cases}
\displaystyle N(\sigma), & \sigma \in \left( \frac{5}{6} - \frac{q}{2p}, 1 - \frac{q}{2p} \right], \\[8pt]
1, & \sigma \in \left( 1 - \frac{q}{2p}, \frac{5}{6} \right),
\end{cases} \\[20pt]
\label{2.13}&p \ge 3q: \quad
N:= N(\sigma), \quad \sigma \in \left( \frac{5}{6} - \frac{q}{2p}, \frac{5}{6} \right), \\[20pt]
\label{2.14}&3p > q > p: \quad
N:=
\begin{cases}
\displaystyle N(\sigma), & \sigma \in \left( \frac{5}{6} - \frac{p}{2q}, 1 - \frac{p}{2q} \right], \\[8pt]
1, & \sigma \in \left( 1 - \frac{p}{2q}, \frac{5}{6} \right),
\end{cases} \\[20pt]
\label{2.15}&q \ge 3p: \quad
N:= N(\sigma), \quad \sigma \in \left( \frac{5}{6} - \frac{p}{2q}, \frac{5}{6} \right),
\end{align}
such that
\begin{equation}\label{2.16}
0 = \beta_0^* < \beta_1^* < \beta_2^* < \cdots < \beta_{N-1}^* \le 1 < \beta_N^*.
\end{equation}

In light of \eqref{2.10}-\eqref{2.15}, we define an increasing sequence \(\{\beta_n\}_{n=1}^N\) such that
\begin{equation}\label{2.17}
0 = \beta_0 < \beta_1 < \beta_2 < \cdots < \beta_{N-1} < \beta_N := 1,
\end{equation}
satisfying the recursive condition
\begin{equation}\label{2.18}
\beta_n < \begin{cases}
\displaystyle \frac{1}{2(1-\sigma)} \left( \frac{q}{p} + \frac{1}{3}\beta_{n-1} \right), & p \ge q, \\[10pt]
\displaystyle \frac{1}{2(1-\sigma)} \left( \frac{p}{q} + \frac{1}{3}\beta_{n-1} \right), & p < q.
\end{cases}
\end{equation}

The purpose of introducing the sequence $\{\beta_n\}$ is to design a refined decomposition of the boundary layer. We utilize the unified boundary layer thickness $\delta(\varepsilon) = \min\{\varepsilon^p, \varepsilon^q\}$.
\begin{center}
    \includegraphics[width=0.5\textwidth]{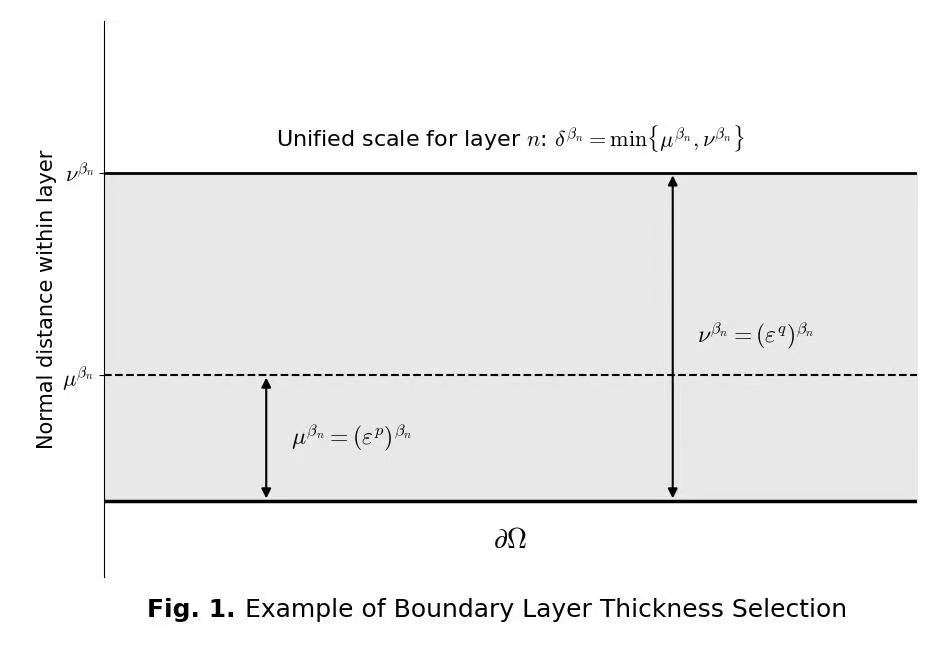}
\end{center}

And we partition the inner region of $\Omega$ into the following subsets:
\begin{equation}
\begin{aligned}\label{2.19}
&V_1 := \Omega^{2\delta^{\beta_1}}, \\
&V_n := \Omega^{2\delta^{\beta_n}} \setminus \Omega^{2\delta^{\beta_{n-1}} + 2\delta^{\beta_n}}, \quad \text{for } 2 \le n \le N,
\end{aligned}
\end{equation}
where
\begin{equation}\label{2.20}
V_{N+1} := \bigg( \bigcup_{n=1}^N V_n \bigg)^c \quad \text{and} \quad V_0 := \varnothing.
\end{equation}
 It is not difficult to check that, for $\varepsilon$ small enough, the Lebesgue measure of these layers and their intersections satisfy:
\begin{equation}\label{2.21}
\operatorname{meas} (V_n) \le C\delta^{\beta_{n-1}}, \quad \operatorname{meas} (V_k \cap V_m) \le
\begin{cases}
C\delta^{\min\{\beta_k,\beta_m\}} & \text{if } |k-m| = 1, \\
0 & \text{if } |k-m| > 1.
\end{cases}
\end{equation}

Subordinate to this decomposition $\{V_n\}_{n=1}^{N}$, we introduce a smooth partition of unity $\{\xi_n\}_{n=1}^{N}$ such that
\begin{equation}\label{2.22}
\operatorname{spt}\xi_n \subset V_n, \quad 0 \le \xi_n \le 1, \quad \sum_{n=1}^{N} \xi_n(x) = 1, \quad \forall x \in \Omega.
\end{equation}
{\bf A key property of this construction is that the gradient vanishes identically on overlapping regions:}
\begin{equation}\label{2.23}
\nabla(\xi_n + \xi_{n+1})^2 = 0 \quad \text{if } x \in V_n \cap V_{n+1}.
\end{equation}
Furthermore for \(n=1,2,\cdots,N\), define
\begin{equation}\label{2.24}
\overline{f}_n=\left\{
\begin{array}{ll}
\displaystyle\int \eta_{\delta^{\beta_n}}(x-y)f(y) dy, \quad & x\in V_n\cap
V^{c}_{n+1},\\
\displaystyle\int \eta_{\delta^{\beta_{n+1}}}(x-y)f(y) dy, \quad & x\in V_n\cap
V_{n+1}.
\end{array}
\right.
\end{equation}
From the construction of the localized regularization, we know that the mollified functions coincide on the same overlaps:
\begin{equation}\label{2.25}
\overline{f}_n = \overline{f}_{n+1} \quad \text{on} \quad V_n \cap V_{n+1}.
\end{equation}

\begin{center}
    \includegraphics[width=0.5\textwidth]{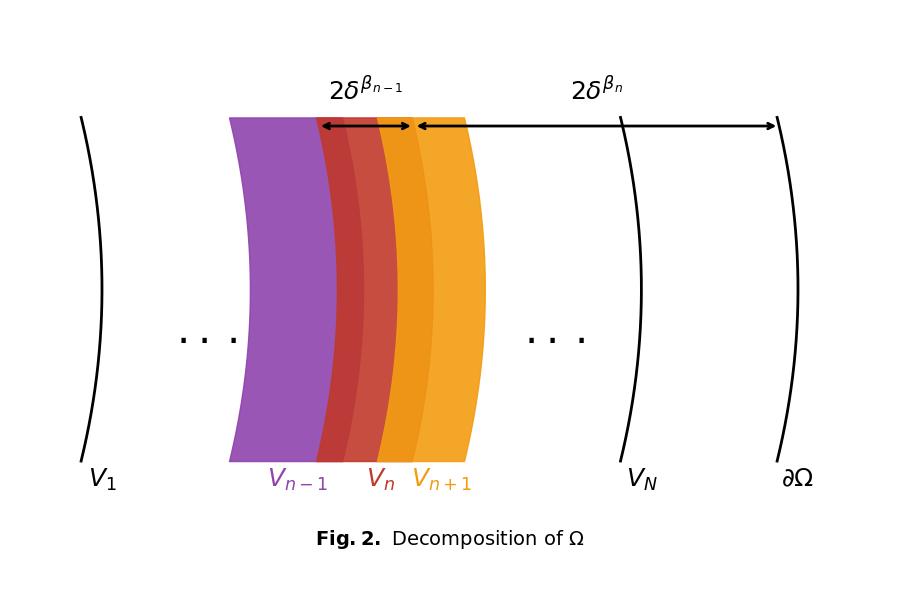}
\end{center}
\begin{Remark}
From geometric viewpoint, the motivation of \eqref{2.19} is to decompose the boundary layer into a sequence of infinitesimal thin, overlapping strips which are parallel to \(\partial\Omega\). It is interesting to observe that such type of decomposition allows locally regularization at each \(V_n\), when these strips approach the boundary. It prevents the mollified functions "escape" outside the boundary \(\partial\Omega\).
\end{Remark}
\begin{Remark}
In contrast with \cite{chen22}, we need to modify the construction when applying the foliation to the MHD system,
due to the presence of the velocity field and magnetic field. In the Navier-Stokes system,  scientists define a sole boundary layer for the velocity field that arises from viscous dissipation. Here, we need to define a unified layer thickness \(\delta(\varepsilon) := \min\{\varepsilon^p, \varepsilon^q\}\) to ensure that both \(\mathbf u\) and \(\mathbf B\) can be mollified at the same foliation. On the other hand, because there are two distinct boundary layers associated with different thicknesses, namely \(\mu \sim \varepsilon^p\) and \(\nu \sim \varepsilon^q\), we need to modify the construction of the sequence $\{\beta_k\}$ to control decay rates.
Last but not least, the coupling terms, like \(\overline{\mathbf{u}^{\varepsilon}}_m \otimes \overline{\mathbf{B}^{\varepsilon}}_m\) and \(\overline{\mathbf{B}^{\varepsilon}}_m \otimes \overline{\mathbf{u}^{\varepsilon}}_m\) in \eqref{3.55}, generate strong nonlinear interactions which brings the essential difficulty to show the convergence of error terms.
\end{Remark}

\section{Proof of Theorem 1.1}
\hspace{1.5em}\noindent Using the boundary layer foliation constructed in Section 2, for each \(N \in \mathbb{N}\) we define the globally regularized fields by
\begin{equation}\label{3.1}
\tilde{\mathbf{u}}^\delta(x,t) = \sum_{n=1}^N \xi_n(x) \overline{\mathbf{u}^{\varepsilon}}_n(x,t), \qquad
\tilde{\mathbf{B}}^\delta(x,t) = \sum_{n=1}^N \xi_n(x) \overline{\mathbf{B}^{\varepsilon}}_n(x,t),
\end{equation}
where $\overline{\mathbf{u}^{\varepsilon}}_n$ and \(\overline{\mathbf{B}^{\varepsilon}}_n\) denote the mollification of \((\mathbf u^\varepsilon, \mathbf B^\varepsilon)\) on the scale \(\delta^{\beta_n}\) associated with the subdomain \(V_n\) defined in \eqref{2.24}. {\bf It should be noted that the mollification defined in \eqref{2.24} differs from that in \eqref{2.1}.}

From (1.1), we have
\begin{equation}
\label{3.2}\left\{
\begin{array}{l}
\partial_t \tilde{\mathbf{u}}^\delta(x,t)  + \sum\limits_{n=1}^N \xi_n \operatorname{div}\overline{(\mathbf{u}^\varepsilon \otimes \mathbf{u}^\varepsilon)}_n-\varepsilon^p \sum\limits_{n=1}^N \xi_n \Delta \overline{\mathbf{u}^{\varepsilon}}_n+ \sum\limits_{n=1}^N \xi_n \nabla\overline{P^{\varepsilon}}_n= \sum\limits_{n=1}^N \xi_n \operatorname{div}\overline{(\mathbf{B}^\varepsilon \otimes \mathbf{B}^\varepsilon)}_n , \\
\partial_t \tilde{\mathbf{B}}^\delta(x,t)  + \sum\limits_{n=1}^N \xi_n \operatorname{div}\overline{(\mathbf{u}^\varepsilon \otimes \mathbf{B}^\varepsilon)}_n-\varepsilon^q \sum\limits_{n=1}^N \xi_n \Delta \overline{\mathbf{B}^{\varepsilon}}_n= \sum\limits_{n=1}^N \xi_n \operatorname{div}\overline{(\mathbf{B}^\varepsilon \otimes \mathbf{u}^\varepsilon)}_n.
\end{array}
\right.
\end{equation}
To localize our analysis away from the physical boundary, we introduce a smooth cut-off function \(\theta(x)\) in \(\Omega\) satisfying
\[
0 \le \theta(x) \le 1,\quad
\theta(x)=1\ \text{for}\ x\in\Omega^{4\delta},\quad
\theta(x)=0\ \text{for}\ x\notin\Omega^{2\delta}\quad\text{and}\quad
|\nabla\theta| \le 4\delta^{-1}.
\]
We test \(\eqref{3.2}_1\) against \(\theta\tilde{\mathbf{u}}^\delta\) and \(\eqref{3.2}_2\) against \(\theta\tilde{\mathbf{B}}^\delta\), then sum and integrate over \(\Omega\times[0,T]\) to obtain the energy equality
\begin{align}\label{3.3}
&\frac{1}{2} \int_\Omega \theta \left( |\tilde{\mathbf{u}}^\delta(x,T)|^2 + |\tilde{\mathbf{B}}^\delta(x,T)|^2 \right) dx
- \frac{1}{2} \int_\Omega \theta \left( |\tilde{\mathbf{u}}^\delta(x,0)|^2 + |\tilde{\mathbf{B}}^\delta(x,0)|^2 \right) dx \notag \\
= &\int_0^T \int_\Omega \theta \Bigl( \sum_{n=1}^N \xi_n \overline{\mathbf{u}^{\varepsilon}}_n \Bigr) \cdot \Bigl( \sum_{n=1}^N \xi_n \varepsilon^p \Delta \overline{\mathbf{u}^{\varepsilon}}_n \Bigr) dx dt  \notag \\
&+ \int_0^T \int_\Omega \theta \Bigl( \sum_{n=1}^N \xi_n \overline{\mathbf{B}^{\varepsilon}}_n \Bigr) \cdot \Bigl( \sum_{n=1}^N \xi_n \varepsilon^q \Delta \overline{\mathbf{B}^{\varepsilon}}_n \Bigr) dx dt \notag \\
&- \int_0^T \int_\Omega \theta \Bigl( \sum_{n=1}^N \xi_n \overline{\mathbf{u}^{\varepsilon}}_n \Bigr) \cdot \Bigl( \sum_{n=1}^N \xi_n \nabla \overline{P^{\varepsilon}}_n \Bigr) dx dt \notag \\
&+ \int_0^T \int_\Omega \theta \Bigl( \sum_{n=1}^N \xi_n \overline{\mathbf{u}^{\varepsilon}}_n \Bigr) \cdot \left(\sum_{n=1}^N \xi_n \operatorname{div} \left(\overline{\left(\mathbf{B}^\varepsilon \otimes \mathbf{B}^\varepsilon\right)}_n-\overline{\left(\mathbf{u}^\varepsilon \otimes \mathbf{u}^\varepsilon\right)} _n\right) \right) dx dt  \notag \\
&+ \int_0^T \int_\Omega \theta \Bigl( \sum_{n=1}^N \xi_n \overline{\mathbf{B}^{\varepsilon}}_n \Bigr) \cdot \left( \sum_{n=1}^N \xi_n \operatorname{div} \left(\overline{\left(\mathbf{B}^\varepsilon \otimes \mathbf{u}^\varepsilon\right)}_n-\overline{\left(\mathbf{u}^\varepsilon \otimes \mathbf{B}^\varepsilon\right)}_n \right) \right) dx dt = \sum_{i=1}^{5} R_i.
\end{align}

In order to prove the main theorem, we need to show that all the terms on the right-hand side of \eqref{3.3} converge to the corresponding weak form of the ideal MHD equations as \(\varepsilon \to 0\). We proceed by analyzing these terms in the subsequent propositions.

\vspace{1em}

\begin{Proposition}[Vanishing of the dissipation error]
Under the assumptions of Theorem 1.1, the dissipation error terms associated with the velocity and magnetic fields vanish in the limit, i.e.,
\begin{equation}\label{3.4}
\lim_{\varepsilon\to0+}\bigl(|R_1|+|R_2|\bigr)=0.
\end{equation}
\end{Proposition}
\begin{proof}
Observing that \(R_1\) and \(R_2\) share the same mathematical structure,
we only present the proof for \(R_2\) below.
Although the argument for \(R_2\) is similar to \cite[Lemma 3.1]{chen22},
the boundary layer scales for the velocity field and the magnetic field are distinct in MHD system,
which requires a subtle analysis on the convergence of several terms.
For completeness, we prove the process in the following,
which relies on the local orthogonality property of the partition of unity \(\{\xi_n\}\):
\begin{equation}
\label{3.5}
\xi_k\xi_m=0\quad\text{whenever }|k-m|\ge2.
\end{equation}
Recalling \eqref{2.21} and integrating by parts then gives
\begin{align}\label{3.6}
\varepsilon^q&\int_0^T\int_{\Omega}\theta\,\Bigl(\sum_{n=1}^N\xi_n\overline{\mathbf{B}^{\varepsilon}}_n\Bigr)\cdot\Bigl(\sum_{n=1}^N\xi_n\Delta\overline{\mathbf{B}^{\varepsilon}}_n\Bigr)\,dx\,dt \notag\\
&=\varepsilon^q\sum_{k,m=1}^N\int_0^T\int_{\Omega}\theta\,\xi_k\xi_m\,\overline{\mathbf{B}^{\varepsilon}}_k\cdot\Delta\overline{\mathbf{B}^{\varepsilon}}_m\,dx\,dt \notag\\
&= -\varepsilon^q\sum_{|k-m|\le1}\int_0^T\int_{\Omega}\theta\,\xi_k\xi_m\,\bigl(\nabla\overline{\mathbf{B}^{\varepsilon}}_k\bigr) : \bigl(\nabla\overline{\mathbf{B}^{\varepsilon}}_m\bigr)\,dx\,dt \notag\\
 &\quad -\varepsilon^q\sum_{|k-m|\le1}\int_0^T\int_{\Omega}\xi_k\xi_m\, \bigl(\nabla\theta\otimes\overline{\mathbf{B}^{\varepsilon}}_k\bigr) : \bigl(\nabla\overline{\mathbf{B}^{\varepsilon}}_m\bigr)dx\,dt \notag\\
&\quad -\varepsilon^q\sum_{|k-m|\le1}\int_0^T\int_{\Omega}
\theta\bigl(\nabla(\xi_k\xi_m)\otimes\overline{\mathbf{B}^{\varepsilon}}_k\bigr) : \bigl(\nabla\overline{\mathbf{B}^{\varepsilon}}_m\bigr)\,dx\,dt \notag\\
&= R_{21}+R_{22}+R_{23}.
\end{align}
\begin{itemize}
    \item \textbf{Estimate for $R_{21}$}
\end{itemize}

Due to the \eqref{2.19}, when \(|k-m|=0\), we have
\begin{align}\label{3.7}
& \left|\varepsilon^q \int_0^T \int_\Omega \theta \xi_k^2 \bigl(\nabla\overline{\mathbf{B}^{\varepsilon}}_k\bigr) : \bigl(\nabla\overline{\mathbf{B}^{\varepsilon}}_m\bigr) \, dxdt \right|\notag \\
&= \left|\varepsilon^q \int_0^T \left( \int_{V_k \cap V_{k+1}} + \int_{V_k \cap V_{k+1}^c} \right) \theta \xi_k^2 \left| \nabla \overline{\mathbf{B}^{\varepsilon}}_k \right|^2 dxdt \right| \notag\\
&\le C \varepsilon^q \int_0^T \left(  \left\| \nabla \overline{\mathbf{B}^{\varepsilon}}_k \right\|_{L^3(V_k \cap V_{k+1})}^2 \left\| \xi_k^2 \right\|_{L^3(V_k \cap V_{k+1})}+\left\| \nabla \overline{\mathbf{B}^{\varepsilon}}_k \right\|_{L^3(V_k \cap V_{k+1}^c)}^2  \left\| \xi_k^2 \right\|_{L^3(V_k \cap V_{k+1}^c)} \right) dt \notag\\
&\le C \varepsilon^q \left( \delta^{ 2\beta_{k+1}(\alpha_2-1)+\frac{1}{3}\beta_{k}} + \delta^{2\beta_k(\alpha_2-1)+\frac{1}{3}\beta_{k-1}} \right) \int_0^T \left\| \mathbf{B}^\varepsilon \right\|_{B_3^{\alpha_2,\infty}(\Omega^\delta)}^2 dt,
\end{align}
where we use \eqref{1.5}, \eqref{2.4}, \eqref{2.21} and \eqref{2.22}.

When \(|k-m|=1\), we use H\"older inequality and \eqref{2.4}-\eqref{2.5} to obtain.
\begin{align}\label{3.8}
& \left|\varepsilon^q \int_0^T \int_\Omega \theta \xi_k \xi_m \bigl(\nabla\overline{\mathbf{B}^{\varepsilon}}_k\bigr) : \bigl(\nabla\overline{\mathbf{B}^{\varepsilon}}_m\bigr) \, dxdt \right| \notag\\
=& \varepsilon^q \left|\int_0^T \int_\Omega \theta \xi_k \xi_{k+1} \bigl(\nabla\overline{\mathbf{B}^{\varepsilon}}_k\bigr) : \bigl(\nabla\overline{\mathbf{B}^{\varepsilon}}_{k+1}\bigr)  dxdt + \int_0^T \int_\Omega \theta \xi_k \xi_{k-1} \bigl(\nabla\overline{\mathbf{B}^{\varepsilon}}_k\bigr) : \bigl(\nabla\overline{\mathbf{B}^{\varepsilon}}_{k-1}\bigr)  dxdt \right| \notag\\
\le& C \varepsilon^q \int_0^T\left\| \nabla \overline{\mathbf{B}^{\varepsilon}}_{k} \right\|_{L^3(V_k \cap V_{k+1})} \left\| \nabla \overline{\mathbf{B}^{\varepsilon}}_{k+1} \right\|_{L^3(V_k \cap V_{k+1})} \left\| \xi_k \xi_{k+1} \right\|_{L^3(V_k \cap V_{k+1})}dt \notag\\
&\quad + C \varepsilon^q \int_0^T \left\| \nabla \overline{\mathbf{B}^{\varepsilon}}_{k} \right\|_{L^3(V_k \cap V_{k-1})} \left\| \nabla \overline{\mathbf{B}^{\varepsilon}}_{k-1} \right\|_{L^3(V_k \cap V_{k-1})} \left\| \xi_k \xi_{k-1} \right\|_{L^3(V_k \cap V_{k-1})}dt \notag\\
\le& C \varepsilon^q \left( \delta^{2\beta_{k+1}(\alpha_2-1)+\frac{1}{3}\beta_k} + \delta^{2\beta_k(\alpha_2-1)+\frac{1}{3}\beta_{k-1}} \right) \int_0^T \left\| \mathbf{B}^\varepsilon \right\|_{B_3^{\alpha_2,\infty}(\Omega^\delta)}^2 dt.
\end{align}

Therefore, in order to obtain that \eqref{3.7} and \eqref{3.8} converge to zero, we need to show that $\varepsilon^q \delta^{2\beta_{k+1}(\alpha_2-1)+\frac{1}{3}\beta_k}$ vanishes as \(\varepsilon \to 0\), and both \eqref{3.7} and \eqref{3.8} converge to zero similarly. Since the values of $p$ and $q$ affect the choice of $\delta$, we distinguish two cases.

\medskip
\textbf{Case I:} Assume \(p \ge q\). Hence, we have \(\varepsilon^p \le \varepsilon^q\). Then we choose \(\delta = \varepsilon^p\), and get
\begin{equation}\label{3.9}
\varepsilon^q \cdot \delta^{2\beta_{k+1}(\alpha_2-1) + \frac{1}{3}\beta_k} = \varepsilon^{p\left(2\beta_{k+1}(\alpha_2-1) + \frac{1}{3}\beta_k + \frac{q}{p}\right)}.
\end{equation}
By the construction of the sequence \(\{\beta_k\}\) given in \eqref{2.18}, we have
\begin{equation}\label{3.10}
2\beta_{k+1}(\alpha_2-1) + \frac{1}{3}\beta_k + \frac{q}{p}
\ge 2\beta_{k+1}(\sigma-1) + \frac{1}{3}\beta_k + \frac{q}{p} > 0.
\end{equation}
Hence, as \(\varepsilon \to 0\), it follows that
\begin{equation}\label{3.11}
\varepsilon^{p\left(2\beta_{k+1}(\alpha_2-1) + \frac{1}{3}\beta_k + \frac{q}{p}\right)} \to 0.
\end{equation}

\textbf{Case II:} Assume \(p < q\). Hence, we have \(\varepsilon^p > \varepsilon^q\).  Choosing \(\delta = \varepsilon^q\) we obtain
\begin{equation}\label{3.12}
\varepsilon^q \cdot \delta^{2\beta_{k+1}(\alpha_2-1) + \frac{1}{3}\beta_k} = \varepsilon^{q\left(2\beta_{k+1}(\alpha_2-1) + \frac{1}{3}\beta_k + 1\right)}.
\end{equation}
From condition \eqref{2.18}, we have
\begin{equation}\label{3.13}
2\beta_{k+1}(\alpha_2-1) + \frac{1}{3}\beta_k + 1 > 2\beta_{k+1}(\sigma-1) + \frac{1}{3}\beta_k + \frac{p}{q} > 0.
\end{equation}
Therefore, as \(\varepsilon \to 0\), we conclude that
\begin{equation}\label{3.14}
\varepsilon^{q\left(2\beta_{k+1}(\alpha_2-1) + \frac{1}{3}\beta_k + 1\right)} \to 0.
\end{equation}

Combining the above estimates for each case, we conclude that
\begin{equation}\label{3.15}
\lim_{\varepsilon\to0+}R_{21}=0.
\end{equation}

\begin{itemize}
    \item \textbf{Estimate for $R_{22}$}
\end{itemize}

Recalling that $\nabla\theta = 0$ for $x\notin\Omega^{2\delta}\cap\Gamma_{4\delta}$,
and the definition of $V_N$, it is evident to observe that $V_N$ contains the region where $\nabla\theta\neq 0$,
so that the integral can be considered only over $V_N$.
For small $\delta$, we use \eqref{1.3}, \eqref{1.5}, \eqref{2.7} and \eqref{2.9} to obtain
\begin{align}\label{3.16}
\left|R_{22}\right|=&\left|\varepsilon^q\sum_{|k-m|\le1}
\int_0^T\int_{\Omega}\xi_k\xi_m\,
\bigl(\nabla\theta\otimes\overline{\mathbf{B}^{\varepsilon}}_k\bigr) :
\bigl(\nabla\overline{\mathbf{B}^{\varepsilon}}_m\bigr)dx\,dt\right|\notag\\
\le & \varepsilon^q \int_0^T \int_{\Omega^{2\delta} \cap \Gamma_{4\delta}} |\nabla\overline{\mathbf{B}^\varepsilon}_N | \, |\nabla\theta|\, \xi_N^2 \left|\overline{\mathbf{B}^\varepsilon}_N - \mathbf{B}^\varepsilon + \mathbf{B}^\varepsilon \right| dxdt \notag\\
\le & \varepsilon^q \int_0^T \int_{\Omega^{2\delta} \cap \Gamma_{4\delta}} |\nabla\overline{\mathbf{B}^\varepsilon}_N | \, |\nabla\theta|\, \xi_N^2\left| \overline{\mathbf{B}^\varepsilon}_N - \mathbf{B}^\varepsilon \right| dxdt + \varepsilon^q \int_0^T \int_{\Omega^{2\delta} \cap \Gamma_{4\delta}} |\nabla\overline{\mathbf{B}^\varepsilon}_N | \, \xi_N^2
|\nabla\theta| |\mathbf{B}^\varepsilon| dxdt \notag\\
\le & \left( \varepsilon^q \int_0^T \int_{\Omega^{2\delta} \cap \Gamma_{4\delta}} |\nabla\overline{\mathbf{B}^\varepsilon}_N |^2 dxdt \right)^{\frac{1}{2}}
\cdot \left( \varepsilon^q \int_0^T \int_{\Omega^{2\delta} \cap \Gamma_{4\delta}} |\nabla\theta|^2 \xi_N^4 \left|\overline{\mathbf{B}^\varepsilon}_N - \mathbf{B}^\varepsilon \right|^2 dxdt \right)^{\frac{1}{2}} \notag\\
&+ \left( \varepsilon^q \int_0^T \int_{\Omega^{2\delta} \cap \Gamma_{4\delta}} |\nabla\overline{\mathbf{B}^\varepsilon}_N |^2 dxdt \right)^{\frac{1}{2}} \cdot \left( \varepsilon^q \int_0^T \int_{\Omega^{2\delta} \cap \Gamma_{4\delta}} \xi_N^4 |\nabla\theta|^2\, |\mathbf{B}^\varepsilon|^2 dxdt \right)^{\frac{1}{2}} \notag\\
\le &\, C \left( \varepsilon^q \int_0^T \int_\Omega |\nabla\overline{\mathbf{B}^\varepsilon}_N |^2 dxdt \right)^{\frac{1}{2}} \cdot \left( \varepsilon^q \int_0^T \delta^{-2} \|\xi_N^4\|_{L^3(\Omega^{2\delta} \cap \Gamma_{4\delta})} \left\| \overline{\mathbf{B}^\varepsilon}_N - \mathbf{B}^\varepsilon \right\|_{L^3(\Omega^{2\delta} \cap \Gamma_{4\delta})}^2 dt \right)^{\frac{1}{2}} \notag\\
&\quad + \left( \varepsilon^q \int_0^T \int_\Omega |\nabla\overline{\mathbf{B}^\varepsilon}_N |^2 dxdt \right)^{\frac{1}{2}} \cdot \left( \varepsilon^q \int_0^T \left\| \nabla\theta \mathbf{B}^\varepsilon \right\|_{L^2(\Omega^{2\delta} \cap \Gamma_{4\delta})}^2 dt \right)^{\frac{1}{2}} \notag\\
\le&\, C\left( \varepsilon^q \int_0^T \int_\Omega |\nabla\overline{\mathbf{B}^\varepsilon}_N |^2 dxdt \right)^{\frac{1}{2}} \cdot \left( \varepsilon^q \int_0^T \delta^{2\alpha_2 - \frac{5}{3}} \left\| \mathbf{B}^\varepsilon \right\|_{B_3^{\alpha_2,\infty}(\Gamma_{4\delta})}^2 dt \right)^{\frac{1}{2}} \notag\\
&\quad + \left(\varepsilon^q \int_0^T \int_\Omega |\nabla \overline{\mathbf{B}^\varepsilon}_N |^2 dxdt  \right)^{\frac{1}{2}} \cdot \left( \varepsilon^q \int_0^T \int_{\Gamma_{4\delta}} |\nabla\overline{\mathbf{B}^\varepsilon}_N |^2 dt \right)^{\frac{1}{2}} \notag\\
\le&\, C \left[ \left( \varepsilon^q \cdot \delta^{2\alpha_2 - \frac{5}{3}} \right)^{\frac{1}{2}} + \left( \varepsilon^q \int_0^T \int_{\Gamma_{4\delta}} |\nabla\overline{\mathbf{B}^\varepsilon}_N |^2 dxdt \right)^{\frac{1}{2}} \right].
\end{align}
Now we show that both the quantity $\varepsilon^q \cdot \delta^{2\alpha_2 - \frac{5}{3}}$  tends to $0$ as $\varepsilon \to 0+$ by considering two cases as follows.
\medskip

\textbf{Case I ($p \ge q$):}
In this case, $\varepsilon^p \le \varepsilon^q$ and $\delta = \varepsilon^p$. Then we have
\begin{equation}\label{3.17}
\varepsilon^q \cdot \delta^{2\alpha_2 - \frac{5}{3}}
= \varepsilon^{q + p \left(2\alpha_2 - \frac{5}{3}\right)}.
\end{equation}
Since $\alpha_2 > \frac{5}{6} - \frac{q}{2p}$,  the exponent
\begin{equation}\label{3.18}
q + p \left(2\alpha_2 - \frac{5}{3}\right) > 0,
\end{equation}
which implies
\begin{equation}\label{3.19}
\lim_{\varepsilon\to 0+} \varepsilon^q \cdot \delta^{2\alpha_2 - \frac{5}{3}} = 0.
\end{equation}

\textbf{Case II ($p < q$):}
Here, $\varepsilon^p > \varepsilon^q$ and $\delta = \varepsilon^q$. Then we have
\begin{equation}\label{3.20}
\varepsilon^q \cdot \delta^{2\alpha_2 - \frac{5}{3}} = \varepsilon^{2q \left(\alpha_2 - \frac{1}{3}\right)}.
\end{equation}
It is obvious \(\alpha_2 > \frac{5}{6} - \frac{p}{2q}>\frac 1 3\) due to \(p<q\). Thus, we obtain
\begin{equation}\label{3.21}
\lim_{\varepsilon\to0+}\varepsilon^q \cdot \delta^{2\alpha_2 - \frac{5}{3}}=0.
\end{equation}
On the other hand, by assumption \eqref{1.8}, the term $\varepsilon^q \int_0^T \int_{\Gamma_{4\delta}} |\nabla \overline{\mathbf{B}^\varepsilon}_N |^2 dxdt$ tend to zero as $\varepsilon \to 0$. Therefore,
\begin{equation}\label{3.22}
\lim_{\varepsilon\to0+}R_{22}=0.
\end{equation}


\begin{itemize}
    \item \textbf{Estimate for $R_{23}$}
\end{itemize}

Finally, we consider the third term in \eqref{3.6}, which involves the gradient of the product $\xi_k\xi_m$. Using \eqref{2.23} and \eqref{2.24}, and according to the supports of the partition functions, we split the sum and rewrite it as
\begin{align}\label{3.23}
R_{23}=&\sum_{|k-m|\le 1}\varepsilon^q\int_0^T\int_{\Omega}\theta\bigl(\nabla(\xi_k\xi_m)\otimes\overline{\mathbf{B}^{\varepsilon}}_k\bigr) : \bigl(\nabla\overline{\mathbf{B}^{\varepsilon}}_m\bigr)\,dx\,dt\notag \\
=&\sum_{k=1}^{N}\varepsilon^q\int_0^T\Bigl(\int_{V_k\cap V_{k-1}}+\int_{V_k\cap V_{k+1}}\Bigr)
\theta\bigl(\nabla(\xi_k^2)\otimes\overline{\mathbf{B}^{\varepsilon}}_k\bigr) : \bigl(\nabla\overline{\mathbf{B}^{\varepsilon}}_k\bigr)\,dx\,dt \nonumber\\
&+\sum_{k=1}^{N-1}\varepsilon^q\int_0^T\int_{V_k\cap V_{k+1}}\theta\left[\bigl(\nabla(\xi_k\xi_{k+1})\otimes\overline{\mathbf{B}^{\varepsilon}}_k\bigr) : \bigl(\nabla\overline{\mathbf{B}^{\varepsilon}}_{k+1}\bigr)+\bigl(\nabla(\xi_k\xi_{k+1})\otimes\overline{\mathbf{B}^{\varepsilon}}_{k+1}\bigr) : \bigl(\nabla\overline{\mathbf{B}^{\varepsilon}}_k\bigr)
\right]\,dx\,dt \nonumber\\
=&\sum_{k=1}^{N}\varepsilon^q\int_0^T\Bigl(\int_{V_k\cap V_{k-1}}\theta\bigl(\nabla(\xi_k^2)\otimes\overline{\mathbf{B}^{\varepsilon}}_k\bigr) : \bigl(\nabla\overline{\mathbf{B}^{\varepsilon}}_k\bigr)\,dx
+\int_{V_k\cap V_{k+1}}\theta\bigl(\nabla(\xi_k^2)\otimes\overline{\mathbf{B}^{\varepsilon}}_{k+1}\bigr) : \bigl(\nabla\overline{\mathbf{B}^{\varepsilon}}_{k+1}\bigr)\,dx\Bigr)dt \nonumber\\
&+\sum_{k=1}^{N-1}\varepsilon^q\int_0^T\int_{V_k\cap V_{k+1}}2\theta\bigl(\nabla(\xi_k\xi_{k+1})\otimes\overline{\mathbf{B}^{\varepsilon}}_{k+1}\bigr) : \bigl(\nabla\overline{\mathbf{B}^{\varepsilon}}_{k+1}\bigr)\,dx\,dt\nonumber\\
=&\sum_{k=1}^{N-1}\varepsilon^q\int_0^T\int_{V_k\cap V_{k+1}}\theta\bigl(\nabla(\xi_{k+1}^2+\xi_k^2+2\xi_k\xi_{k+1})\otimes\overline{\mathbf{B}^{\varepsilon}}_{k+1}\bigr) : \bigl(\nabla\overline{\mathbf{B}^{\varepsilon}}_{k+1}\bigr)\,dx\,dt\nonumber\\
=&\sum_{k=1}^{N-1}\varepsilon^q\int_0^T\int_{V_k\cap V_{k+1}}\theta\left(\nabla\left((\xi_{k}+\xi_{k+1})^2\right)\otimes\overline{\mathbf{B}^{\varepsilon}}_{k+1}\right) : \bigl(\nabla\overline{\mathbf{B}^{\varepsilon}}_{k+1}\bigr)\,dx\,dt=0,
\end{align}
where \eqref{2.23} is used at the last step.

This completes the estimate for the magnetic error term \(R_2\).
An analogous argument, together with the viscous part of condition \eqref{1.8}, yields
\begin{equation}\label{3.24}
\lim_{\varepsilon\to0+}R_1=0.
\end{equation}

Therefore, we complete the proof of Proposition 3.1.
\end{proof}

\begin{Proposition}[Vanishing of Pressure]
Under the assumptions of Theorem 1.1, we have the pressure term vanishes in the limit as:
\begin{equation}\label{3.25}
\lim_{\varepsilon\to0+}R_3=0.
\end{equation}
\end{Proposition}
\begin{proof}
Recalling incompressibility condition \(\operatorname{div}\overline{\mathbf{u}^{\varepsilon}}_n=0\), and integrating \(R_3\) by parts, we get
\begin{align}\label{3.26}
R_3
&= \int_0^T \int_\Omega \theta \Bigl( \sum_{n=1}^N \xi_n \overline{\mathbf{u}^{\varepsilon}}_n \Bigr) \cdot \Bigl( \sum_{n=1}^N \xi_n \nabla \overline{P^{\varepsilon}}_n \Bigr) dx dt\notag\\
&= \sum_{|k-m|\le1}\int_0^T\int_{\Omega}\theta\xi_k\xi_m\overline{\mathbf{u}^{\varepsilon}}_k\cdot\nabla\overline{P^\varepsilon}_m\,dx\,dt
= -\sum_{|k-m|\le1}\int_0^T\int_{\Omega}\operatorname{div}\bigl(\theta\xi_k\xi_m\overline{\mathbf{u}^{\varepsilon}}_k\bigr)\overline{P^\varepsilon}_m\,dx\,dt\notag\\
&= -\sum_{|k-m|\le1}\int_0^T\int_{\Omega}\xi_k\xi_m\nabla\theta\cdot\overline{\mathbf{u}^{\varepsilon}}_k\overline{P^\varepsilon}_m\,dx\,dt
  -\sum_{|k-m|\le1}\int_0^T\int_{\Omega}\theta\nabla(\xi_k\xi_m)\cdot\overline{\mathbf{u}^{\varepsilon}}_k\overline{P^\varepsilon}_m\,dx\,dt.
\end{align}

We first consider the case \(|k-m|=0\). Expanding the gradient of the product and using \eqref{2.23} yields
\begin{align}\label{3.27}
&\sum_{|k-m|=0}\int_0^T\int_{\Omega}
\xi_k\xi_m\nabla\theta\cdot\overline{\mathbf{u}^{\varepsilon}}_k\overline{P^\varepsilon}_m\,dx\,dt
+\sum_{|k-m|=0}\int_0^T\int_{\Omega}\theta\nabla(\xi_k\xi_m)\cdot\overline{\mathbf{u}^{\varepsilon}}_k\overline{P^\varepsilon}_m\,dx\,dt\notag\\
=&\sum_{k=1}^N\int_0^T\int_{V_k}\overline{P^\varepsilon}_k\xi_k^2\overline{\mathbf{u}^{\varepsilon}}_k\cdot\nabla\theta\,dx\,dt
+\sum_{k=1}^N\int_0^T\int_{V_k}\theta\overline{P^\varepsilon}_k\overline{\mathbf{u}^{\varepsilon}}_k\cdot\nabla(\xi_k^2)\,dx\,dt\notag\\[4pt]
=&\sum_{k=1}^N\int_0^T\int_{V_k}\overline{P^\varepsilon}_k\xi_k^2\overline{\mathbf{u}^{\varepsilon}}_k\cdot\nabla\theta\,dx\,dt+\sum_{k=1}^N\int_0^T\int_{V_{k-1}\cap V_k}\theta\overline{P}_k^\varepsilon\overline{\mathbf{u}^{\varepsilon}}_k\cdot\nabla(\xi_k^2)\,dx\,dt\notag\\
&+\sum_{k=1}^N\int_0^T\int_{V_k\cap V_{k+1}}\theta\overline{P^\varepsilon}_{k+1}\overline{\mathbf{u}^\varepsilon}_{k+1}\cdot\nabla(\xi_k^2)\,dx\,dt\notag\\[4pt]
=&\sum_{k=1}^N\int_0^T\int_{V_k}\overline{P^\varepsilon}_k\xi_k^2\overline{\mathbf{u}^{\varepsilon}}_k\cdot\nabla\theta\,dx\,dt
-2\sum_{k=1}^{N-1}\int_0^T\int_{V_k\cap V_{k+1}}\theta\overline{P^\varepsilon}_k\overline{\mathbf{u}^{\varepsilon}}_k\cdot\nabla(\xi_k\xi_{k+1})\,dx\,dt.
\end{align}

Next we treat the case \(|k-m|=1\). As similar computation in
\eqref{3.23}, it follows that
\begin{align}\label{3.28}
&\sum_{|k-m|=1}\int_0^T\int_{\Omega}
\xi_k\xi_m\nabla\theta\cdot\overline{\mathbf{u}^{\varepsilon}}_k\overline{P^\varepsilon}_m\,dx\,dt
+\sum_{|k-m|=1}\int_0^T\int_{\Omega}\theta\nabla(\xi_k\xi_m)\cdot\overline{\mathbf{u}^{\varepsilon}}_k\overline{P^\varepsilon}_m\,dx\,dt\notag\\
=&\sum_{k=1}^{N-1}2\int_0^T\int_{V_k\cap V_{k+1}}\Bigl(\overline{P^\varepsilon}_k\xi_k\xi_{k+1}
\overline{\mathbf{u}^\varepsilon}_{k}\cdot\nabla\theta\,
+\theta\overline{P^\varepsilon}_k\overline{\mathbf{u}^\varepsilon}_{k}\cdot\nabla(\xi_k\xi_{k+1})\Bigr)dx\,dt.
\end{align}

Combining \eqref{3.27} and \eqref{3.28}, we obtain the following estimate for the sum over \(|k-m|\le1\):
\begin{align}\label{3.29}
\left|R_3\right|
&=\left|\sum_{k=1}^N\int_0^T\int_{V_k}\overline{P^\varepsilon}_k\xi_k^2\overline{\mathbf{u}^{\varepsilon}}_k\cdot\nabla\theta\,dx\,dt
+2\sum_{k=1}^{N-1}\int_0^T\int_{V_k\cap V_{k+1}}\xi_k\xi_{k+1}\,\overline{P^\varepsilon}_k\overline{\mathbf{u}^{\varepsilon}}_k\cdot\nabla\theta\,dx\,dt\right|\notag\\[4pt]
&\le C\int_0^T\int_{\Gamma_{4\delta}\cap\Omega^{2\delta}}|\overline{P^\varepsilon}_k||\overline{\mathbf{u}^{\varepsilon}}_k||\nabla\theta|\,dx\,dt\notag\\[4pt]
&\le C\int_0^T\int_{\Gamma_{4\delta}\cap\Omega^{2\delta}}|\overline{P^\varepsilon}_k||\nabla\mathbf{u}^\varepsilon|\,dx\,dt\notag\\[4pt]
&\le C\Bigl(\int_0^T\|\overline{P^\varepsilon}_k\|_{L^\infty(\Gamma_{4\delta})}^2\,dt\Bigr)^{\frac12}
\Bigl(\delta\int_0^T\int_{\Gamma_{4\delta}}|\nabla\mathbf{u}^\varepsilon|^2\,dx\,dt\Bigr)^{\frac12}.
\end{align}
Therefore by \eqref{1.7} and \eqref{1.8}, we conclude
\begin{equation}\label{3.30}
\lim_{\varepsilon\to0+}R_3=0.
\end{equation}
\end{proof}

\begin{Proposition}
[Vanishing of Nonlinear Energy Flux]
Let $(\mathbf{u}^\varepsilon, \mathbf{B}^\varepsilon)$ satisfy the uniform spatial H\"older bounds in the boundary layer. Then the nonlinear energy flux vanishes,
\begin{equation}\label{3.31}
\lim_{\varepsilon\to0+}\bigl(|R_4|+|R_5|\bigr)=0.
\end{equation}
\end{Proposition}
\begin{proof}
Firstly, we integrate $R_4+R_5$ by parts, and rewrite
\begin{equation}\label{3.32}
R_4 + R_5=  \sum_{i=1}^{5} I_i,
\end{equation}
where the nonlinear terms are explicitly given by
\begin{align}
\label{3.33}I_1
&= -\sum_{|k-m|\le1}\int_0^T\int_{\Omega}
\bigl(\overline{\mathbf{u}^{\varepsilon}}_m\otimes\overline{\mathbf{u}^{\varepsilon}}_m
-\overline{(\mathbf{u}^\varepsilon\otimes\mathbf{u}^\varepsilon)}_m\big):
\nabla\bigl(\theta\xi_k\xi_m\overline{\mathbf{u}^{\varepsilon}}_k\bigr)\,dx\,dt,
\\[4pt]
\label{3.34}I_2
&= \sum_{|k-m|\le1}\int_0^T\int_{\Omega}
\bigl(\overline{\mathbf{B}^{\varepsilon}}_m\otimes\overline{\mathbf{B}^{\varepsilon}}_m
-\overline{(\mathbf{B}^\varepsilon\otimes\mathbf{B}^\varepsilon)}_m\big):
\nabla\bigl(\theta\xi_k\xi_m\overline{\mathbf{u}^{\varepsilon}}_k\bigr)\,dx\,dt,
\\[4pt]
\label{3.35}I_3
&= -\sum_{|k-m|\le1}\int_0^T\int_{\Omega}
\bigl(\overline{\mathbf{u}^{\varepsilon}}_m\otimes\overline{\mathbf{B}^{\varepsilon}}_m
-\overline{(\mathbf{u}^\varepsilon\otimes\mathbf{B}^\varepsilon)}_m\big):
\nabla\bigl(\theta\xi_k\xi_m\overline{\mathbf{B}^{\varepsilon}}_k\bigr)\,dx\,dt,
\\[4pt]
\label{3.36}I_4
&= \sum_{|k-m|\le1}\int_0^T\int_{\Omega}
\bigl(\overline{\mathbf{B}^{\varepsilon}}_m\otimes\overline{\mathbf{u}^{\varepsilon}}_m
-\overline{(\mathbf{B}^\varepsilon\otimes\mathbf{u}^\varepsilon)}_m\big):
\nabla\bigl(\theta\xi_k\xi_m\overline{\mathbf{B}^{\varepsilon}}_k\bigr)\,dx\,dt,
\\[4pt]
\label{3.37}I_5
&= \sum_{|k-m| \le 1} \int_0^T \int_\Omega \left( \overline{\mathbf{u}^{\varepsilon}}_m \otimes \overline{\mathbf{u}^{\varepsilon}}_m - \overline{\mathbf{B}^{\varepsilon}}_m \otimes \overline{\mathbf{B}^{\varepsilon}}_m \right) : \nabla \bigl( \theta \xi_k \xi_m \overline{\mathbf{u}^{\varepsilon}}_k \bigr) \, dx dt\notag
\\
&\quad + \sum_{|k-m| \le 1} \int_0^T \int_\Omega \left(\overline{\mathbf{u}^{\varepsilon}}_m \otimes \overline{\mathbf{B}^{\varepsilon}}_m-\overline{\mathbf{B}^{\varepsilon}}_m \otimes \overline{\mathbf{u}^{\varepsilon}}_m \right) : \nabla \bigl( \theta \xi_k \xi_m \overline{\mathbf{B}^{\varepsilon}}_k \bigr) \, dx dt.
\end{align}

In contrast with the Navier-Stokes system which only involves velocity field, the MHD system requires handling the coupling between velocity and magnetic fields, which is more complicated to control. In order to show the convergence of all the terms in \eqref{3.32}, we proceed with the proof step by step as follows:

\bigskip
\noindent \textbf{Step 1: (Velocity Advection Commutator)}
\bigskip

Firstly, we consider the commutator error arising from the convective term, that is \eqref{3.33} and obtain
\begin{equation}\label{3.38}
\lim_{\varepsilon\to0+}\sum_{|k-m|\le1}\int_0^T\int_{\Omega}
\bigl(\overline{\mathbf{u}^{\varepsilon}}_m\otimes\overline{\mathbf{u}^{\varepsilon}}_m-\overline{(\mathbf{u}^\varepsilon\otimes\mathbf{u}^\varepsilon)}_m\bigr):
\nabla\bigl(\theta\xi_k\xi_m\overline{\mathbf{u}^{\varepsilon}}_k\bigr)\,dx\,dt=0.
\end{equation}
The detailed proof can be found in Lemma 3.2 of \cite{chen22}, and we skip it.

\bigskip
\noindent \textbf{Step 2: (Magnetic Stress Commutator)}

\bigskip
Using a similar argument, we estimate the commutator term \eqref{3.34} associated with the Lorentz force. By replacing the velocity field with the magnetic field in the commutator structure, we obtain
\begin{equation}\label{3.39}
 \lim_{\varepsilon \to 0+} \sum_{|k-m| \le 1} \int_0^T \int_\Omega \bigl( \overline{\mathbf{B}^{\varepsilon}}_m \otimes \overline{\mathbf{B}^{\varepsilon}}_m - \overline{(\mathbf{B}^\varepsilon \otimes \mathbf{B}^\varepsilon)}_m \bigr) : \nabla \bigl( \theta \xi_k \xi_m \overline{\mathbf{u}^{\varepsilon}}_k \bigr) \, dx dt = 0.
\end{equation}

We apply the Leibniz product rule to the gradient term and decompose it into three parts as follows:
\begin{align}\label{3.40}
&\sum_{|k-m|\le1}\int_0^T\int_{\Omega}
\bigl(\overline{\mathbf{B}^{\varepsilon}}_m\otimes\overline{\mathbf{B}^{\varepsilon}}_m
-\overline{(\mathbf{B}^\varepsilon\otimes\mathbf{B}^\varepsilon)}_m\bigr):
\nabla\bigl(\theta\xi_k\xi_m\overline{\mathbf{u}^{\varepsilon}}_k\bigr)\,dx\,dt \notag\\
=&\sum_{|k-m|\le1}\int_0^T\int_{\Omega}
\bigl(\overline{\mathbf{B}^{\varepsilon}}_m\otimes\overline{\mathbf{B}^{\varepsilon}}_m
-\overline{(\mathbf{B}^\varepsilon\otimes\mathbf{B}^\varepsilon)}_m\bigr):
\bigl(\xi_k\xi_m\nabla\theta \otimes \overline{\mathbf{u}^{\varepsilon}}_k\bigr)\,dx\,dt  \notag\\
&+ \sum_{|k-m|\le1}\int_0^T\int_{\Omega}
\bigl(\overline{\mathbf{B}^{\varepsilon}}_m\otimes\overline{\mathbf{B}^{\varepsilon}}_m
-\overline{(\mathbf{B}^\varepsilon\otimes\mathbf{B}^\varepsilon)}_m\bigr):
\bigl(\theta\xi_k\xi_m\nabla\overline{\mathbf{u}^{\varepsilon}}_k\bigr)\,dx\,dt  \notag\\
&+ \sum_{|k-m|\le1}\int_0^T\int_{\Omega}
\bigl(\overline{\mathbf{B}^{\varepsilon}}_m\otimes\overline{\mathbf{B}^{\varepsilon}}_m
-\overline{(\mathbf{B}^\varepsilon\otimes\mathbf{B}^\varepsilon)}_m\bigr):
\bigl(\theta\nabla(\xi_k\xi_m) \otimes \overline{\mathbf{u}^{\varepsilon}}_k\bigr)\,dx\,dt  \notag\\
=& I_{21} + I_{22} + I_{23}.
\end{align}

Firstly, we concern on $I_{21}$. Applying Hardy inequality \eqref{2.9} together with the scaling properties of the cutoff function $\theta$, we obtain
\begin{align}\label{3.41}
|I_{21}|=
&\left| \sum_{|k-m|\leq 1} \int_0^T \int_\Omega \left( \overline{\mathbf{B}^{\varepsilon}}_m \otimes \overline{\mathbf{B}^{\varepsilon}}_m - \overline{({\mathbf{B}}^\varepsilon \otimes {\mathbf{B}}^\varepsilon)}_m \right) : \left( \xi_k\xi_m \nabla \theta\otimes\overline{\mathbf{u}^{\varepsilon}}_k \right) dx dt \right|\notag \\
&\leq C \left( \int_0^T \int_{\Gamma_{4\delta} \cap \Omega^{2\delta}} |\overline{\mathbf{B}^\varepsilon}_N |^4 dx dt \right)^{\frac{1}{2}} \left( \int_0^T \int_{\Gamma_{4\delta} \cap \Omega^{2\delta}} |{\xi^2_N}\nabla \theta\otimes (\overline{\mathbf{u}^\varepsilon}_N - \mathbf{u}^\varepsilon + \mathbf{u}^\varepsilon)|^2 dx dt \right)^{\frac{1}{2}}\notag \\
&\leq C \left(\delta \int_0^T \|\overline{\mathbf{B}^\varepsilon}_N\|_{L^\infty(\Gamma_{4\delta})}^4 dt \right)^{\frac{1}{2}} \left( \int_0^T \delta^{-\frac{5}{3}} \|\overline{\mathbf{u}^\varepsilon}_N - \mathbf{u}^\varepsilon\|_{L^3(\Gamma_{4\delta} \cap \Omega^{2\delta})}^2 dt + \int_0^T \int_{\Gamma_{4\delta}} |\nabla \theta \mathbf{u}^\varepsilon|^2 dx dt \right)^{\frac{1}{2}}\notag \\
&\leq C \delta^{\frac{1}{2}} \left( \int_0^T \delta^{-\frac{5}{3}} \cdot \delta^{2\alpha_1} \|\mathbf{u}^\varepsilon\|_{B^{\alpha_1, \infty}_3(\Omega^\delta)}^2 dt + \int_0^T \int_{\Gamma_{4\delta}} |\nabla \mathbf{u}^\varepsilon|^2 dx dt \right)^{\frac{1}{2}}\notag \\
&\leq C \left( \delta^{2\alpha_1-\frac{2}{3}} + \delta \int_0^T \int_{\Gamma_{4\delta}} |\nabla \mathbf{u}^\varepsilon|^2 dx dt \right)^{\frac{1}{2}}.
\end{align}

We could deduce the fact that $\alpha_1 > 1/3$ from \eqref{1.6}, no matter $p\geq p$ or $p<q$.
Therefore, combining $\alpha_1 > 1/3$ together with condition \eqref{1.8}, we obtain
\begin{equation}\label{3.42}
\lim_{\varepsilon\to 0+} I_{21} = 0.
\end{equation}

Next, we turn to the term \(I_{22}\). For the case \(|k-m|=0\), we focus on the diagonal contribution:
\begin{align}\label{3.43}
&\sum_{|k-m|=0} \int_0^T \int_\Omega \left( \overline{\mathbf{B}^{\varepsilon}}_m \otimes \overline{\mathbf{B}^{\varepsilon}}_m - \overline{(\mathbf{B}^{\varepsilon} \otimes \mathbf{B}^{\varepsilon})}_m \right) : \left( \theta \xi_k \xi_m \nabla\overline{\mathbf{u}^{\varepsilon}}_k \right) dxdt \notag\\
=& \sum_{k=1}^{N} \int_0^T \int_\Omega \left( \overline{\mathbf{B}^{\varepsilon}}_k \otimes \overline{\mathbf{B}^{\varepsilon}}_k - \overline{(\mathbf{B}^{\varepsilon} \otimes \mathbf{B}^{\varepsilon})}_k \right) :\left( \theta \xi_k^2 \nabla \overline{\mathbf{u}^{\varepsilon}}_k\right) \, dxdt\notag \\
\leq &C \sum_{k=1}^{N} \int_0^T \left\| \overline{\mathbf{B}^{\varepsilon}}_k \otimes \overline{\mathbf{B}^{\varepsilon}}_k - \overline{(\mathbf{B}^{\varepsilon} \otimes \mathbf{B}^{\varepsilon})}_k \right\|_{L^{\frac{3}{2}}(V_k \cap V_{k+1})} \|\nabla \overline{\mathbf{u}^{\varepsilon}}_k\|_{L^3(V_k \cap V_{k+1})} dt \notag\\
&+ C \sum_{k=1}^{N} \int_0^T \left\| \overline{\mathbf{B}^{\varepsilon}}_k \otimes \overline{\mathbf{B}^{\varepsilon}}_k - \overline{(\mathbf{B}^{\varepsilon} \otimes \mathbf{B}^{\varepsilon})}_k \right\|_{L^{\frac{3}{2}}(V_k \cap V_{k+1}^c)} \|\nabla \overline{\mathbf{u}^{\varepsilon}}_k\|_{L^3(V_k \cap V_{k+1}^c)} dt \notag\\
\leq & C \sum_{k=1}^{N} \int_0^T \left( \delta^{\left[2\beta_{k+1}\alpha_2+\beta_{k+1}(\alpha_1-1)\right]} + \delta^{\left[2\beta_k\alpha_2+\beta_k(\alpha_1-1)\right]} \right) \|\mathbf{B}^\varepsilon\|^2_{B_3^{\alpha_2,\infty}(\Omega^\delta)} \|\mathbf{\mathbf{u}}^\varepsilon\|_{B_3^{\alpha_1,\infty}(\Omega^\delta)} dt \notag\\
\leq & C \sum_{k=1}^{N} \delta^{\beta_k(2\alpha_2+\alpha_1-1)}.
\end{align}

Next, for the off-diagonal case \(|k-m|=1\), we have
\begin{align}\label{3.44}
&\sum_{|k-m|=1} \int_0^T \int_\Omega \left( \overline{\mathbf{B}^{\varepsilon}}_m \otimes \overline{\mathbf{B}^{\varepsilon}}_m - \overline{(\mathbf{B}^{\varepsilon} \otimes \mathbf{B}^{\varepsilon})}_m \right) : \left( \theta \xi_k \xi_m \nabla \overline{\mathbf{u}^{\varepsilon}}_k \right) dxdt \notag\\
=& \sum_{k=1}^{N} \int_0^T \int_{V_k \cap V_{k+1}} \left( \overline{\mathbf{B}^{\varepsilon}}_{k+1} \otimes \overline{\mathbf{B}^{\varepsilon}}_{k+1} - \overline{(\mathbf{B}^{\varepsilon} \otimes \mathbf{B}^{\varepsilon})}_{k+1} \right) : \left(\theta \xi_k \xi_{k+1} \nabla \overline{\mathbf{u}^{\varepsilon}}_k\right) \, dxdt\notag \\
&+ \sum_{k=1}^{N} \int_0^T \int_{V_k \cap V_{k-1}} \left( \overline{\mathbf{B}^{\varepsilon}}_{k-1} \otimes \overline{\mathbf{B}^{\varepsilon}}_{k-1} - \overline{(\mathbf{B}^{\varepsilon} \otimes \mathbf{B}^{\varepsilon})}_{k-1} \right) : \left(\theta \xi_k \xi_{k-1} \nabla \overline{\mathbf{u}^{\varepsilon}}_k\right) \, dxdt\notag \\
 \leq& C \sum_{k=1}^{N} \int_0^T \left\| \overline{\mathbf{B}^{\varepsilon}}_{k+1} \otimes \overline{\mathbf{B}^{\varepsilon}}_{k+1} - \overline{(\mathbf{B}^{\varepsilon} \otimes \mathbf{B}^{\varepsilon})}_{k+1} \right\|_{L^{\frac{3}{2}}(V_k \cap V_{k+1})} \|\nabla \overline{\mathbf{u}^{\varepsilon}}_k\|_{L^3(V_k \cap V_{k+1})} dt \notag\\
&+ C \sum_{k=1}^{N} \int_0^T \left\| \overline{\mathbf{B}^{\varepsilon}}_{k-1} \otimes \overline{\mathbf{B}^{\varepsilon}}_{k-1} - \overline{(\mathbf{B}^{\varepsilon} \otimes \mathbf{B}^{\varepsilon})}_{k-1} \right\|_{L^{\frac{3}{2}}(V_k \cap V_{k-1})} \|\nabla \overline{\mathbf{u}^{\varepsilon}}_k\|_{L^3(V_k \cap V_{k-1})} dt \notag\\
 \leq& C \sum_{k=1}^{N} \int_0^T \left( \delta^{\left[2\beta_{k+1}\alpha_2+\beta_{k+1}(\alpha_1-1)\right]} + \delta^{\left[2\beta_{k}\alpha_2+\beta_k(\alpha_1-1)\right]} \right) \|\mathbf{B}^\varepsilon\|^2_{B_3^{\alpha_2,\infty}(\Omega^\delta)} \|\mathbf{u}^\varepsilon\|_{B_3^{\alpha_1,\infty}(\Omega^\delta)} dt\notag \\
 \leq& C \sum_{k=1}^N \delta^{\beta_k (2\alpha_2 + \alpha_1 - 1)}.
\end{align}
From \eqref{1.6}, we get $\alpha_1,\alpha_2>\frac{1}{3}$, which implies $2\alpha_2 + \alpha_1 - 1 > 0$. So we obtain
\begin{equation}\label{3.45}
\lim_{\varepsilon\to 0+}  I_{22}= 0.
\end{equation}
Finally we estimate \(I_{23}\). Using \eqref{2.23}, we have
\begin{align}\label{3.46}
I_{23}
=& \sum_{|k-m|\le 1} \int_0^T \int_\Omega \left( \overline{\mathbf{B}^{\varepsilon}}_m \otimes \overline{\mathbf{B}^{\varepsilon}}_m - \overline{\left( \mathbf{B}^{\varepsilon} \otimes \mathbf{B}^{\varepsilon}\right)} _m \right): \left( \theta\nabla( \xi_k \xi_m)\otimes \overline{\mathbf{u}^{\varepsilon}}_k \right) dxdt \notag\\
=& \sum_{k=1}^{N} \int_0^T \Bigl(\int_{V_k\cap V_{k-1}} + \int_{V_k\cap V_{k+1}}\Bigr) \left( \overline{\mathbf{B}^{\varepsilon}}_k \otimes \overline{\mathbf{B}^{\varepsilon}}_k -  \overline{\left(\mathbf{B}^{\varepsilon} \otimes \mathbf{B}^{\varepsilon}\right)} _k \right) : \left( \theta\nabla(\xi_k^2) \otimes \overline{\mathbf{u}^{\varepsilon}}_k \right) dxdt\notag \\
&+ \sum_{k=1}^{N} \int_0^T \int_{V_k\cap V_{k+1}} \left( \overline{\mathbf{B}^{\varepsilon}}_{k+1} \otimes \overline{\mathbf{B}^{\varepsilon}}_{k+1} - \overline{\left( \mathbf{B}^{\varepsilon} \otimes \mathbf{B}^{\varepsilon}\right)}_{k+1} \right) : \left( \theta\nabla(\xi_k \xi_{k+1})\otimes  \overline{\mathbf{u}^{\varepsilon}}_k \right) dxdt \notag\\
&+ \sum_{k=1}^{N} \int_0^T \int_{V_k\cap V_{k-1}} \left( \overline{\mathbf{B}^{\varepsilon}}_{k-1} \otimes \overline{\mathbf{B}^{\varepsilon}}_{k-1} - \overline{\left( \mathbf{B}^{\varepsilon} \otimes \mathbf{B}^{\varepsilon} \right)}_{k-1} \right) : \left( \theta\nabla(\xi_k \xi_{k-1})\otimes  \overline{\mathbf{u}^{\varepsilon}}_k \right) dxdt\notag \\
=& \sum_{k=1}^{N} \int_0^T \Bigl(\int_{V_k\cap V_{k-1}} \left( \overline{\mathbf{B}^{\varepsilon}}_k \otimes \overline{\mathbf{B}^{\varepsilon}}_k - \overline{\left( \mathbf{B}^{\varepsilon} \otimes \mathbf{B}^{\varepsilon} \right)}_k \right) : \left( \theta\nabla(\xi_k^2) \times\overline{\mathbf{u}^{\varepsilon}}_k \right) dx\notag \\
&\qquad+ \int_{V_k\cap V_{k+1}} \left( \overline{\mathbf{B}^{\varepsilon}}_{k+1} \otimes \overline{\mathbf{B}^{\varepsilon}}_{k+1} - \overline{\left( \mathbf{B}^{\varepsilon} \otimes \mathbf{B}^{\varepsilon} \right)}_{k+1} \right) : \left( \theta\nabla(\xi_k^2) \otimes\overline{\mathbf{u}^{\varepsilon}}_k \right) dx \Bigr) dt\notag \\
&+ \sum_{k=1}^{N-1} \int_0^T \int_{V_k\cap V_{k+1}} \left( \overline{\mathbf{B}^{\varepsilon}}_{k+1} \otimes \overline{\mathbf{B}^{\varepsilon}}_{k+1} - \overline{\left( \mathbf{B}^{\varepsilon} \otimes \mathbf{B}^{\varepsilon} \right)}_{k+1} \right) : \left( \theta\nabla(2\xi_k \xi_{k+1}) \otimes\overline{\mathbf{u}^{\varepsilon}}_k \right) dxdt \notag\\
=& \sum_{k=1}^{N-1} \int_0^T \int_{V_k\cap V_{k+1}} \left( \overline{\mathbf{B}^{\varepsilon}}_{k+1} \otimes \overline{\mathbf{B}^{\varepsilon}}_{k+1} - \overline{\left(\mathbf{B}^{\varepsilon} \otimes \mathbf{B}^{\varepsilon}\right)}_{k+1} \right) : \left(\theta\left(\nabla(\xi_{k+1}^2) + \nabla(\xi_k^2) + \nabla(2\xi_k \xi_{k+1}) \right)\otimes \overline{\mathbf{u}^{\varepsilon}}_k\right) dxdt\notag \\
=& \sum_{k=1}^{N-1} \int_0^T \int_{V_k \cap V_{k+1}} \left( \overline{\mathbf{B}^{\varepsilon}}_{k+1} \otimes \overline{\mathbf{B}^{\varepsilon}}_{k+1} -  \overline{\left(\mathbf{B}^{\varepsilon} \otimes \mathbf{B}^{\varepsilon}\right)} _{k+1} \right): \left( \theta \nabla \left( \left(\xi_k + \xi_{k+1} \right)^2\right)\otimes \overline{\mathbf{u}^{\varepsilon}}_k \right) dxdt= 0.
\end{align}

Combining \eqref{3.42}, \eqref{3.45} and \eqref{3.46}, we get the conclusion \eqref{3.39}.
\bigskip

\noindent \textbf{Step 3: (Induction Advection Commutator)}
\medskip

For the mixed commutator term \(I_3\) in the induction equation, the regularity conditions (1.6) guarantee that
\begin{equation}\label{3.47}
    \lim_{\varepsilon \to 0+} I_3 = 0.
\end{equation}
By Leibniz's rule, we decompose the integral in \eqref{3.35} into $I_{31}, I_{32}, I_{33}$, where $\nabla$ acts on $\theta$, $\overline{\mathbf{B}^{\varepsilon}}_k$ and $\xi_k\xi_m$ respectively:
\begin{align}\label{3.48}
&\sum_{|k-m|\le1}\int_0^T\int_{\Omega}
\bigl(\overline{\mathbf{u}^{\varepsilon}}_m\otimes\overline{\mathbf{B}^{\varepsilon}}_m
-\overline{(\mathbf{u}^\varepsilon\otimes\mathbf{B}^\varepsilon)}_m\bigr):
\nabla\bigl(\theta\xi_k\xi_m\overline{\mathbf{B}^{\varepsilon}}_k\bigr)\,dx\,dt \notag\\
=& \sum_{|k-m|\le1}\int_0^T\int_{\Omega}
\bigl(\overline{\mathbf{u}^{\varepsilon}}_m\otimes\overline{\mathbf{B}^{\varepsilon}}_m
-\overline{(\mathbf{u}^\varepsilon\otimes\mathbf{B}^\varepsilon)}_m\bigr):
\bigl(\xi_k\xi_m\nabla\theta \otimes\overline{\mathbf{B}^{\varepsilon}}_k\bigr)\,dx\,dt \notag\\
&+ \sum_{|k-m|\le1}\int_0^T\int_{\Omega}
\bigl(\overline{\mathbf{u}^{\varepsilon}}_m\otimes\overline{\mathbf{B}^{\varepsilon}}_m
-\overline{(\mathbf{u}^\varepsilon\otimes\mathbf{B}^\varepsilon)}_m\bigr):
\bigl(\theta\xi_k\xi_m  \nabla\overline{\mathbf{B}^{\varepsilon}}_k\bigr)\,dx\,dt\notag \\
&+ \sum_{|k-m|\le1}\int_0^T\int_{\Omega}
\bigl(\overline{\mathbf{u}^{\varepsilon}}_m\otimes\overline{\mathbf{B}^{\varepsilon}}_m
-\overline{(\mathbf{u}^\varepsilon\otimes\mathbf{B}^\varepsilon)}_m\bigr):
\bigl(\theta \nabla(\xi_k\xi_m) \otimes \overline{\mathbf{B}^{\varepsilon}}_k\bigr)\,dx\,dt \notag\\
=& I_{31} + I_{32} + I_{33}.
\end{align}

Starting with the term containing $\nabla \theta$ (i.e., $I_{31}$), we have
\begin{align}\label{3.49}
& \sum_{|k-m|\le 1} \int_0^T \int_\Omega \left( \overline{\mathbf{u}^{\varepsilon}}_m \otimes \overline{\mathbf{B}^{\varepsilon}}_m - \overline{(\mathbf{u}^\varepsilon \otimes \mathbf{B}^\varepsilon)}_m \right) : \left(  \xi_k \xi_m\nabla\theta\otimes \overline{\mathbf{B}^{\varepsilon}}_k \right) dxdt \notag\\
=& \int_0^T \int_{\Omega^{2\delta} \cap \Gamma_{4\delta}} \left( \overline{\mathbf{u}^\varepsilon}_N \otimes \overline{\mathbf{B}^\varepsilon}_N - \overline{\left(\mathbf{u}^\varepsilon \otimes \mathbf{B}^\varepsilon\right)}_N \right) : \left(\xi_N^2 \nabla \theta \otimes\overline{\mathbf{B}^\varepsilon}_N \right) dxdt\notag \\
\leq & \int_0^T \left\| \overline{\mathbf{u}^\varepsilon}_N \otimes \overline{\mathbf{B}^\varepsilon}_N - \overline{(\mathbf{u}^\varepsilon \otimes \mathbf{B}^\varepsilon)}_N \right\|_{L^{\frac{3}{2}}(\Omega^{2\delta} \cap \Gamma_{4\delta})} \left\| \nabla \theta \, \overline{\mathbf{B}^\varepsilon}_N \right\|_{L^3(\Omega^{2\delta} \cap \Gamma_{4\delta})} dt\notag \\
\leq& C \int_0^T \delta^{\alpha_1+\alpha_2} \left\| \mathbf{u}^{\varepsilon} \right\|_{B_3^{\alpha_1,\infty}(\Omega^\delta)} \left\| \mathbf{B}^{\varepsilon}\right\|_{B_3^{\alpha_2,\infty}(\Omega^\delta)} \left\| \nabla \overline{\mathbf{B}^\varepsilon}_N \right\|_{L^3(\Omega^{2\delta} \cap \Gamma_{4\delta})} dt \notag\\
\leq& C\delta^{\alpha_1+2\alpha_2-1} \int_0^T \left\|\mathbf{u}^{\varepsilon}\right\|_{B_3^{\alpha_1,\infty}(\Omega^\delta)} \left\| \mathbf{B}^{\varepsilon} \right\|_{B_3^{\alpha_2,\infty}(\Omega^\delta)}^2 dt \notag\\
\leq& C\delta^{\alpha_1+2\alpha_2-1}.
\end{align}

Due to the regularity conditions on $\alpha_1$ and $\alpha_2$ in \eqref{1.6}, we have $\alpha_1 + 2\alpha_2 - 1 > 0$. Hence $I_{31} \to 0$ as $\varepsilon\to0+$.

Next we estimate $I_{32}$. For the case $|k-m|=0$, we have
\begin{align}\label{3.50}
&\sum_{k=1}^N \int_0^T \int_\Omega \bigl( \overline{\mathbf{u}^{\varepsilon}}_k \otimes \overline{\mathbf{B}^{\varepsilon}}_k - (\overline{{\mathbf{u}}^\varepsilon \otimes {\mathbf{B}}^\varepsilon)}_k \bigr) : \left(\theta \xi_k^2 \nabla \overline{\mathbf{B}^{\varepsilon}}_k\right)
\, dx \, dt\notag \\
\leq &C \sum_{k=1}^N \int_0^T \left\| \overline{\mathbf{u}^{\varepsilon}}_k \otimes \overline{\mathbf{B}^{\varepsilon}}_k - \overline{\left({\mathbf{u}}^\varepsilon \otimes {\mathbf{B}}^\varepsilon\right)}_k \right\|_{L^{\frac32}(V_k \cap V_{k+1})} \left\| \nabla \overline{\mathbf{B}^{\varepsilon}}_k \right\|_{L^3(V_k \cap V_{k+1})} \, dt \notag\\
&+ C \sum_{k=1}^N \int_0^T \left\| \overline{\mathbf{u}^{\varepsilon}}_k \otimes \overline{\mathbf{B}^{\varepsilon}}_k - \overline{\left({\mathbf{u}}^\varepsilon \otimes {\mathbf{B}}^\varepsilon\right)}_k \right\|_{L^{\frac32}(V_k \cap V_{k+1}^c)} \left\| \nabla \overline{\mathbf{B}^{\varepsilon}}_k \right\|_{L^3(V_k \cap V_{k+1}^c)} \, dt \notag\\
\leq &C \sum_{k=1}^N \int_0^T \left( \delta^{\beta_{k+1}(\alpha_1+\alpha_2) + \beta_{k+1}(\alpha_2-1)} + \delta^{\beta_k(\alpha_1+\alpha_2) + \beta_k(\alpha_2-1)} \right) \|\mathbf{B}^\varepsilon\|^2_{B_3^{\alpha_2,\infty}(\Omega^\delta)} \|\mathbf{u}^\varepsilon\|_{B_3^{\alpha_1,\infty}(\Omega^\delta)} \, dt \notag\\
\leq& C \sum_{k=1}^N \delta^{\beta_k (\alpha_1 + 2\alpha_2 - 1)}.
\end{align}

Similarly, for the case $|k-m|=1$, we obtain
\begin{align}\label{3.51}
&\sum_{|k-m|=1} \int_0^T \int_\Omega \left( \overline{\mathbf{u}^{\varepsilon}}_m \otimes \overline{\mathbf{B}^{\varepsilon}}_m - \overline{(\mathbf{u}^\varepsilon \otimes \mathbf{B}^\varepsilon)}_m \right) : \left(\theta \xi_k \xi_m \nabla \overline{\mathbf{B}^{\varepsilon}}_k\right) \, dxdt \\
= &\sum_{k=1}^N \int_0^T \left( \overline{\mathbf{u}^\varepsilon}_{k+1}\otimes \overline{\mathbf{B}^\varepsilon}_{k+1}- \overline{(\mathbf{u}^\varepsilon \otimes \mathbf{B}^\varepsilon)}_{k+1} \right) : \left(\theta \xi_k \xi_{k+1} \nabla \overline{\mathbf{B}^{\varepsilon}}_k\right)\, dxdt \notag\\
& + \sum_{k=1}^N \int_0^T \left( \overline{\mathbf{u}^\varepsilon}_{k-1}\otimes \overline{\mathbf{B}}_{k-1}^\varepsilon - \overline{(\mathbf{u}^\varepsilon \otimes \mathbf{B}^\varepsilon)}_{k-1} \right) :\left( \theta \xi_{k-1} \xi_k \nabla \overline{\mathbf{B}^{\varepsilon}}_k\right) \, dxdt \notag\\
\le& C \sum_{k=1}^N \int_0^T \left\| \overline{\mathbf{u}^\varepsilon}_{k+1} \otimes \overline{\mathbf{B}^\varepsilon}_{k+1} - \overline{(\mathbf{u}^\varepsilon \otimes \mathbf{B}^\varepsilon)}_{k+1} \right\|_{L^{\frac{3}{2}}(V_k \cap V_{k+1})} \left\| \nabla \overline{\mathbf{B}^\varepsilon}_{k+1} \right\|_{L^3(V_k \cap V_{k+1})} dt \notag\\
& + \sum_{k=1}^N \int_0^T \left\| \overline{\mathbf{u}^\varepsilon}_{k-1} \otimes \overline{\mathbf{B}^\varepsilon}_{k-1} - \overline{(\mathbf{u}^\varepsilon \otimes \mathbf{B}^\varepsilon)}_{k-1} \right\|_{L^{\frac{3}{2}}(V_k \cap V_{k-1})} \left\| \nabla \overline{\mathbf{B}^\varepsilon}_{k-1} \right\|_{L^3(V_k \cap V_{k-1})} dt \notag\\
\le &C \sum_{k=1}^N \delta^{\beta_{k+1}(\alpha_1+\alpha_2) + \beta_{k+1}(\alpha_2-1)} + \sum_{k=1}^N \delta^{\beta_{k}(\alpha_1+\alpha_2) + \beta_{k}(\alpha_2-1)}\notag\\
\le &C \sum_{k=1}^N \delta^{\beta_{k}(\alpha_1+\alpha_2) + \beta_{k}(\alpha_2-1)}.
\end{align}
Because $2\alpha_2+\alpha_1-1>0$,  thus $\delta^{\beta_k (\alpha_1 + 2\alpha_2 - 1)}$  tends to $0$ and hence $I_{32} \to 0$ as $\varepsilon\to0+$.

Finally for $I_{33}$, the proof of convergence is the same as that of \eqref{3.23} and \eqref{3.46} via condition \eqref{2.23}. Thus
\begin{align}\label{3.53}
&\sum_{|k-m| \le 1} \int_0^T \int_\Omega \bigl( \overline{\mathbf{u}^{\varepsilon}}_m \otimes \overline{\mathbf{B}^{\varepsilon}}_m - \overline{\left({\mathbf{u}}^\varepsilon \otimes {\mathbf{B}}^\varepsilon\right)}_m \bigr) : \left(\theta \nabla(\xi_k \xi_m) \times\overline{\mathbf{B}^{\varepsilon}}_k\right) dx dt \notag\\
= & \sum_{k=1}^{N-1} \int_0^T \int_{V_k \cap V_{k+1}} \bigl( \overline{\mathbf{u}^\varepsilon}_{k+1} \otimes \overline{\mathbf{B}^\varepsilon}_{k+1} - \overline{\left({\mathbf{u}}^\varepsilon \otimes {\mathbf{B}}^\varepsilon\right)}_{k+1} \bigr) : \left(\theta \nabla \bigl( \xi_k^2 + \xi_{k+1}^2 + 2\xi_k \xi_{k+1} \bigr)\otimes \overline{\mathbf{B}^{\varepsilon}}_k\right) dx dt = 0.
\end{align}
We omit the details.

In summary, we have established the following convergence \eqref{3.47} result for the cross-interaction terms,
provided the regularity exponents $\alpha_1$ and $\alpha_2$ satisfy the conditions in \eqref{1.6}.

\bigskip
\noindent\textbf{Step 4: (Induction stretching commutator).}
\medskip

Under the regularity hypotheses of Theorem 1.1, provided the exponents satisfy \eqref{1.6},
we have
\begin{equation}\label{3.54}
\lim_{\varepsilon\to0+} I_4= 0.
\end{equation}
The proof follows exactly the same process as Step 3, with only the roles of $\mathbf{u}$ and $\mathbf{B}$ interchanged in the nonlinearity, hence we omit the detail here.

\bigskip
\noindent \textbf{Step 5: (Direct Product Terms).}
\bigskip

We now turn our attention to the intricate nonlinear term $I_5$. By applying integration by parts to \eqref{3.37} and utilizing the divergence-free conditions $\nabla \cdot \overline{\mathbf{u}^{\varepsilon}}_m = 0$ and $\nabla \cdot \overline{\mathbf{B}^{\varepsilon}}_m = 0$, we can directly rewrite $I_5$ in the convective form as follows:
\begin{align}\label{3.55}
I_5 =& - \sum_{|k-m| \le 1} \int_0^T \int_\Omega \theta \xi_k \xi_m \overline{\mathbf{u}^{\varepsilon}}_k \cdot \bigl( (\overline{\mathbf{u}^{\varepsilon}}_m \cdot \nabla) \overline{\mathbf{u}^{\varepsilon}}_m \bigr) \, dx dt \notag\\
&+ \sum_{|k-m| \le 1} \int_0^T \int_\Omega \theta \xi_k \xi_m \overline{\mathbf{u}^{\varepsilon}}_k \cdot \bigl( (\overline{\mathbf{B}^{\varepsilon}}_m \cdot \nabla) \overline{\mathbf{B}^{\varepsilon}}_m \bigr) \, dx dt \notag\\
& + \sum_{|k-m| \le 1} \int_0^T \int_\Omega \theta \xi_k \xi_m \overline{\mathbf{B}^{\varepsilon}}_k \cdot \bigl( (\overline{\mathbf{B}^{\varepsilon}}_m \cdot \nabla) \overline{\mathbf{u}^{\varepsilon}}_m \bigr) \, dx dt \notag\\
& - \sum_{|k-m| \le 1} \int_0^T \int_\Omega \theta \xi_k \xi_m \overline{\mathbf{B}^{\varepsilon}}_k \cdot \bigl( (\overline{\mathbf{u}^{\varepsilon}}_m \cdot \nabla) \overline{\mathbf{B}^{\varepsilon}}_m \bigr) \, dx dt\notag \\
=& - \sum_{|k-m| \le 1} \int_0^T \int_\Omega \theta \xi_k \xi_m \, \Biggl\{
\overline{\mathbf{u}^{\varepsilon}}_k \cdot \Bigl[ (\overline{\mathbf{u}^{\varepsilon}}_m \cdot \nabla) \overline{\mathbf{u}^{\varepsilon}}_m - (\overline{\mathbf{B}^{\varepsilon}}_m \cdot \nabla) \overline{\mathbf{B}^{\varepsilon}}_m \Bigr]\notag \\
&\qquad\qquad\qquad\qquad\qquad+ \overline{\mathbf{B}^{\varepsilon}}_k \cdot \Bigl[ (\overline{\mathbf{u}^{\varepsilon}}_m \cdot \nabla) \overline{\mathbf{B}^{\varepsilon}}_m - (\overline{\mathbf{B}^{\varepsilon}}_m \cdot \nabla) \overline{\mathbf{u}^{\varepsilon}}_m \Bigr] \Biggr\} \, dx dt.
\end{align}

Firstly, we concern on the diagonal interactions, i.e., $|k-m|=0$, which can be rewritten as:
\begin{align}\label{3.56}
& \sum_{|k-m|=0} \int_0^T \int_\Omega \theta \xi_k \xi_m \Biggl\{
   \overline{\mathbf{u}^{\varepsilon}}_k \cdot \Bigl[ (\overline{\mathbf{u}^{\varepsilon}}_m \cdot \nabla) \overline{\mathbf{u}^{\varepsilon}}_m - (\overline{\mathbf{B}^{\varepsilon}}_m \cdot \nabla) \overline{\mathbf{B}^{\varepsilon}}_m \Bigr] \nonumber \\
&\qquad + \overline{\mathbf{B}^{\varepsilon}}_k \cdot \Bigl[ (\overline{\mathbf{u}^{\varepsilon}}_m \cdot \nabla) \overline{\mathbf{B}^{\varepsilon}}_m - (\overline{\mathbf{B}^{\varepsilon}}_m \cdot \nabla) \overline{\mathbf{u}^{\varepsilon}}_m \Bigr] \Biggr\} \, dx dt \nonumber \\
&= \sum_{m=1}^N \int_0^T \int_\Omega \theta \xi_m^2 \Biggl\{
   \overline{\mathbf{u}^{\varepsilon}}_m \cdot \Bigl[ (\overline{\mathbf{u}^{\varepsilon}}_m \cdot \nabla) \overline{\mathbf{u}^{\varepsilon}}_m - (\overline{\mathbf{B}^{\varepsilon}}_m \cdot \nabla) \overline{\mathbf{B}^{\varepsilon}}_m \Bigr] \nonumber \\
&\qquad + \overline{\mathbf{B}^{\varepsilon}}_m \cdot \Bigl[ (\overline{\mathbf{u}^{\varepsilon}}_m \cdot \nabla) \overline{\mathbf{B}^{\varepsilon}}_m - (\overline{\mathbf{B}^{\varepsilon}}_m \cdot \nabla) \overline{\mathbf{u}^{\varepsilon}}_m \Bigr] \Biggr\} \, dx dt.
\end{align}

In order to control the nonlinear couple terms in \eqref{3.56}, we introduce the localized averaged Els\"asser variables:
\begin{align*}
\overline{\mathbf{W}^\varepsilon}_{+,m} = \overline{\mathbf{u}^{\varepsilon}}_m + \overline{\mathbf{B}^{\varepsilon}}_m, \qquad \overline{\mathbf{W}^\varepsilon}_{-,m} = \overline{\mathbf{u}^{\varepsilon}}_m - \overline{\mathbf{B}^{\varepsilon}}_m.
\end{align*}

\begin{Remark}

In the theory of hyperbolic partial differential equations, the Els\"asser variables  are precisely the Riemann invariants
associated with the characteristic speeds \(
W_{\pm}\triangleq u\pm B\) in ideal incompressible MHD system \eqref{1.2}.
They satisfy the following equation:
\begin{equation}\label{3.57}
\partial_t W_\pm+W_\mp\cdot\nabla W_\pm=-\nabla P.
\end{equation}
Along the characteristic curves \(\dfrac{dx}{dt} = W_\mp\), the corresponding Els\"asser variable
\(W^\pm\) remains constant in the linear  regime, up to coupling effects. This identification diagonalises the hyperbolic system and separates forward and backward
propagating waves, showing that Els\"asser variables play the same role as Riemann invariants
do in classical gas dynamics or in any first-order hyperbolic system. We suggest readers refer to \cite{cobb} for details.
\end{Remark}

Next, we define the corresponding symmetric nonlinear function for each spatial scale $m$:
\begin{align}
\label{3.58} D_{+,m}^\varepsilon &= \int_0^T \int_\Omega \theta \xi_m^2 \,
\overline{\mathbf{W}^\varepsilon}_{+,m} \cdot \bigl[ (\overline{\mathbf{W}^\varepsilon}_{-,m} \cdot \nabla) \overline{\mathbf{W}^\varepsilon}_{+,m} \bigr] \, dx dt,\\
\label{3.59}D_{-,m}^\varepsilon &= \int_0^T \int_\Omega \theta \xi_m^2 \,
\overline{\mathbf{W}^\varepsilon}_{-,m} \cdot \bigl[ (\overline{\mathbf{W}^\varepsilon}_{+,m} \cdot \nabla) \overline{\mathbf{W}^\varepsilon}_{-,m} \bigr] \, dx dt.
\end{align}
According to the definition of $\mathbf{W}_{+,m}^\varepsilon$, $\mathbf{W}_{-,m}^\varepsilon$, it is evident to find
\begin{align}\label{3.60}
\overline{\mathbf{W}^\varepsilon}_{+,m} \cdot \bigl[ (\overline{\mathbf{W}^\varepsilon}_{-,m} \cdot \nabla) \overline{\mathbf{W}^\varepsilon}_{+,m} \bigr]
=& (\overline{\mathbf{u}^{\varepsilon}}_m + \overline{\mathbf{B}^{\varepsilon}}_m) \cdot \bigl[ ((\overline{\mathbf{u}^{\varepsilon}}_m - \overline{\mathbf{B}^{\varepsilon}}_m) \cdot \nabla) (\overline{\mathbf{u}^{\varepsilon}}_m + \overline{\mathbf{B}^{\varepsilon}}_m) \bigr] \notag\\
=& \overline{\mathbf{u}^{\varepsilon}}_m \cdot (\overline{\mathbf{u}^{\varepsilon}}_m \cdot \nabla) \overline{\mathbf{u}^{\varepsilon}}_m
   + \overline{\mathbf{u}^{\varepsilon}}_m \cdot (\overline{\mathbf{u}^{\varepsilon}}_m \cdot \nabla) \overline{\mathbf{B}^{\varepsilon}}_m\notag \\
& - \overline{\mathbf{u}^{\varepsilon}}_m \cdot (\overline{\mathbf{B}^{\varepsilon}}_m \cdot \nabla) \overline{\mathbf{u}^{\varepsilon}}_m
   - \overline{\mathbf{u}^{\varepsilon}}_m \cdot (\overline{\mathbf{B}^{\varepsilon}}_m \cdot \nabla) \overline{\mathbf{B}^{\varepsilon}}_m \notag\\
&+ \overline{\mathbf{B}^{\varepsilon}}_m \cdot (\overline{\mathbf{u}^{\varepsilon}}_m \cdot \nabla) \overline{\mathbf{u}^{\varepsilon}}_m
   + \overline{\mathbf{B}^{\varepsilon}}_m \cdot (\overline{\mathbf{u}^{\varepsilon}}_m \cdot \nabla) \overline{\mathbf{B}^{\varepsilon}}_m \notag\\
&- \overline{\mathbf{B}^{\varepsilon}}_m \cdot (\overline{\mathbf{B}^{\varepsilon}}_m \cdot \nabla) \overline{\mathbf{u}^{\varepsilon}}_m
   - \overline{\mathbf{B}^{\varepsilon}}_m \cdot (\overline{\mathbf{B}^{\varepsilon}}_m \cdot \nabla) \overline{\mathbf{B}^{\varepsilon}}_m,
\end{align}
and 
\begin{align}\label{3.61}
\overline{\mathbf{W}^\varepsilon}_{-,m} \cdot \bigl[ (\overline{\mathbf{W}^\varepsilon}_{+,m} \cdot \nabla) \overline{\mathbf{W}^\varepsilon}_{-,m} \bigr]
= &(\overline{\mathbf{u}^{\varepsilon}}_m - \overline{\mathbf{B}^{\varepsilon}}_m) \cdot \bigl[ ((\overline{\mathbf{u}^{\varepsilon}}_m + \overline{\mathbf{B}^{\varepsilon}}_m) \cdot \nabla) (\overline{\mathbf{u}^{\varepsilon}}_m - \overline{\mathbf{B}^{\varepsilon}}_m) \bigr] \notag\\
= & \overline{\mathbf{u}^{\varepsilon}}_m \cdot (\overline{\mathbf{u}^{\varepsilon}}_m \cdot \nabla) \overline{\mathbf{u}^{\varepsilon}}_m
   - \overline{\mathbf{u}^{\varepsilon}}_m \cdot (\overline{\mathbf{u}^{\varepsilon}}_m \cdot \nabla) \overline{\mathbf{B}^{\varepsilon}}_m\notag \\
& + \overline{\mathbf{u}^{\varepsilon}}_m \cdot (\overline{\mathbf{B}^{\varepsilon}}_m \cdot \nabla) \overline{\mathbf{u}^{\varepsilon}}_m
   - \overline{\mathbf{u}^{\varepsilon}}_m \cdot (\overline{\mathbf{B}^{\varepsilon}}_m \cdot \nabla) \overline{\mathbf{B}^{\varepsilon}}_m\notag \\
& - \overline{\mathbf{B}^{\varepsilon}}_m \cdot (\overline{\mathbf{u}^{\varepsilon}}_m \cdot \nabla) \overline{\mathbf{u}^{\varepsilon}}_m
   + \overline{\mathbf{B}^{\varepsilon}}_m \cdot (\overline{\mathbf{u}^{\varepsilon}}_m \cdot \nabla) \overline{\mathbf{B}^{\varepsilon}}_m\notag \\
& - \overline{\mathbf{B}^{\varepsilon}}_m \cdot (\overline{\mathbf{B}^{\varepsilon}}_m \cdot \nabla) \overline{\mathbf{u}^{\varepsilon}}_m
   + \overline{\mathbf{B}^{\varepsilon}}_m \cdot (\overline{\mathbf{B}^{\varepsilon}}_m \cdot \nabla) \overline{\mathbf{B}^{\varepsilon}}_m .
\end{align}

Putting \eqref{3.60} and \eqref{3.61} into \eqref{3.56}, we get
\begin{align}\label{3.62}
&\sum_{m=1}^N \int_0^T \int_\Omega \theta \xi_m^2 \Biggl\{
   \overline{\mathbf{u}^{\varepsilon}}_m \cdot \Bigl[ (\overline{\mathbf{u}^{\varepsilon}}_m \cdot \nabla) \overline{\mathbf{u}^{\varepsilon}}_m - (\overline{\mathbf{B}^{\varepsilon}}_m \cdot \nabla) \overline{\mathbf{B}^{\varepsilon}}_m \Bigr] \notag\\
&\qquad + \overline{\mathbf{B}^{\varepsilon}}_m \cdot \Bigl[ (\overline{\mathbf{u}^{\varepsilon}}_m \cdot \nabla) \overline{\mathbf{B}^{\varepsilon}}_m - (\overline{\mathbf{B}^{\varepsilon}}_m \cdot \nabla) \overline{\mathbf{u}^{\varepsilon}}_m \Bigr] \Biggr\} \, dx dt \notag\\
= & \frac{1}{2} \sum_{m=1}^N \bigl( D_{+,m}^\varepsilon + D_{-,m}^\varepsilon \bigr).
\end{align}

Next, we turn to the case $|k-m|=1$. Likewise, we also define
\begin{align}\label{3.63}
D_{+,k,m}^\varepsilon &= \int_0^T \int_\Omega \theta \xi_k \xi_m \overline{\mathbf{W}^\varepsilon}_{+,k} \cdot \bigl[ (\overline{\mathbf{W}^\varepsilon}_{-,m} \cdot \nabla) \overline{\mathbf{W}^\varepsilon}_{+,m} \bigr] dx dt, \\
\label{3.64} D_{-,k,m}^\varepsilon &= \int_0^T \int_\Omega \theta \xi_k \xi_m \overline{\mathbf{W}^\varepsilon}_{-,k} \cdot \bigl[ (\overline{\mathbf{W}^\varepsilon}_{+,m} \cdot \nabla) \overline{\mathbf{W}^\varepsilon}_{-,m} \bigr] dx dt.
\end{align}
By virtue of \eqref{3.63}-\eqref{3.64}, it is straightforward to get the following fact by calculating:
\begin{align}\label{3.65}
\overline{\mathbf{W}^\varepsilon}_{+,k} \cdot \bigl[ (\overline{\mathbf{W}^\varepsilon}_{-,m} \cdot \nabla) \overline{\mathbf{W}^\varepsilon}_{+,m} \bigr]
=& (\overline{\mathbf{u}^{\varepsilon}}_k + \overline{\mathbf{B}^{\varepsilon}}_k) \cdot \bigl[ ((\overline{\mathbf{u}^{\varepsilon}}_m - \overline{\mathbf{B}^{\varepsilon}}_m) \cdot \nabla) (\overline{\mathbf{u}^{\varepsilon}}_m + \overline{\mathbf{B}^{\varepsilon}}_m) \bigr] \notag\\
= &\overline{\mathbf{u}^{\varepsilon}}_k \cdot (\overline{\mathbf{u}^{\varepsilon}}_m \cdot \nabla) \overline{\mathbf{u}^{\varepsilon}}_m
   + \overline{\mathbf{u}^{\varepsilon}}_k \cdot (\overline{\mathbf{u}^{\varepsilon}}_m \cdot \nabla) \overline{\mathbf{B}^{\varepsilon}}_m\notag \\
&- \overline{\mathbf{u}^{\varepsilon}}_k \cdot (\overline{\mathbf{B}^{\varepsilon}}_m \cdot \nabla) \overline{\mathbf{u}^{\varepsilon}}_m
   - \overline{\mathbf{u}^{\varepsilon}}_k \cdot (\overline{\mathbf{B}^{\varepsilon}}_m \cdot \nabla) \overline{\mathbf{B}^{\varepsilon}}_m \notag\\
& + \overline{\mathbf{B}^{\varepsilon}}_k \cdot (\overline{\mathbf{u}^{\varepsilon}}_m \cdot \nabla) \overline{\mathbf{u}^{\varepsilon}}_m
   + \overline{\mathbf{B}^{\varepsilon}}_k \cdot (\overline{\mathbf{u}^{\varepsilon}}_m \cdot \nabla) \overline{\mathbf{B}^{\varepsilon}}_m \notag\\
&- \overline{\mathbf{B}^{\varepsilon}}_k \cdot (\overline{\mathbf{B}^{\varepsilon}}_m \cdot \nabla) \overline{\mathbf{u}^{\varepsilon}}_m
   - \overline{\mathbf{B}^{\varepsilon}}_k \cdot (\overline{\mathbf{B}^{\varepsilon}}_m \cdot \nabla) \overline{\mathbf{B}^{\varepsilon}}_m,
\end{align}
and
\begin{align}\label{3.66}
\overline{\mathbf{W}^\varepsilon}_{-,k} \cdot \bigl[ (\overline{\mathbf{W}^\varepsilon}_{+,m} \cdot \nabla) \overline{\mathbf{W}^\varepsilon}_{-,m} \bigr]
=& (\overline{\mathbf{u}^{\varepsilon}}_k - \overline{\mathbf{B}^{\varepsilon}}_k) \cdot \bigl[ ((\overline{\mathbf{u}^{\varepsilon}}_m + \overline{\mathbf{B}^{\varepsilon}}_m) \cdot \nabla) (\overline{\mathbf{u}^{\varepsilon}}_m - \overline{\mathbf{B}^{\varepsilon}}_m) \bigr] \notag\\
= &\overline{\mathbf{u}^{\varepsilon}}_k \cdot (\overline{\mathbf{u}^{\varepsilon}}_m \cdot \nabla) \overline{\mathbf{u}^{\varepsilon}}_m
   - \overline{\mathbf{u}^{\varepsilon}}_k \cdot (\overline{\mathbf{u}^{\varepsilon}}_m \cdot \nabla) \overline{\mathbf{B}^{\varepsilon}}_m \notag\\
&+ \overline{\mathbf{u}^{\varepsilon}}_k \cdot (\overline{\mathbf{B}^{\varepsilon}}_m \cdot \nabla) \overline{\mathbf{u}^{\varepsilon}}_m
   - \overline{\mathbf{u}^{\varepsilon}}_k \cdot (\overline{\mathbf{B}^{\varepsilon}}_m \cdot \nabla) \overline{\mathbf{B}^{\varepsilon}}_m \notag\\
&- \overline{\mathbf{B}^{\varepsilon}}_k \cdot (\overline{\mathbf{u}^{\varepsilon}}_m \cdot \nabla) \overline{\mathbf{u}^{\varepsilon}}_m
   + \overline{\mathbf{B}^{\varepsilon}}_k \cdot (\overline{\mathbf{u}^{\varepsilon}}_m \cdot \nabla) \overline{\mathbf{B}^{\varepsilon}}_m\notag \\
& - \overline{\mathbf{B}^{\varepsilon}}_k \cdot (\overline{\mathbf{B}^{\varepsilon}}_m \cdot \nabla) \overline{\mathbf{u}^{\varepsilon}}_m
   + \overline{\mathbf{B}^{\varepsilon}}_k \cdot (\overline{\mathbf{B}^{\varepsilon}}_m \cdot \nabla) \overline{\mathbf{B}^{\varepsilon}}_m .
\end{align}
Similarly, putting \eqref{3.65} and \eqref{3.66} into \eqref{3.63}-\eqref{3.64}, we have
\begin{align}\label{3.67}
&\sum_{|k-m| = 1} \int_0^T \int_\Omega \theta \xi_k \xi_m \, \Biggl\{
\overline{\mathbf{u}^{\varepsilon}}_k \cdot \Bigl[ (\overline{\mathbf{u}^{\varepsilon}}_m \cdot \nabla) \overline{\mathbf{u}^{\varepsilon}}_m - (\overline{\mathbf{B}^{\varepsilon}}_m \cdot \nabla) \overline{\mathbf{B}^{\varepsilon}}_m \Bigr] \notag \\
&\qquad\qquad + \overline{\mathbf{B}^{\varepsilon}}_k \cdot \Bigl[ (\overline{\mathbf{u}^{\varepsilon}}_m \cdot \nabla) \overline{\mathbf{B}^{\varepsilon}}_m - (\overline{\mathbf{B}^{\varepsilon}}_m \cdot \nabla) \overline{\mathbf{u}^{\varepsilon}}_m \Bigr] \Biggr\} \, dx dt \notag \\
&=\frac{1}{2} \sum_{|k-m|=1} \bigl( D_{+,k,m}^\varepsilon + D_{-,k,m}^\varepsilon \bigr).
\end{align}

Combining the diagonal contribution \eqref{3.62} and the off-diagonal contribution \eqref{3.67}, the nonlinear term $I_5$ is reconstructed as the sum of the macroscopic Els\"asser energy fluxes:
\begin{align}\label{3.68}
I_5 =& -\frac{1}{2} \sum_{m=1}^N \bigl( D_{+,m}^\varepsilon + D_{-,m}^\varepsilon \bigr) - \frac{1}{2} \sum_{|k-m|=1} \bigl( D_{+,k,m}^\varepsilon + D_{-,k,m}^\varepsilon \bigr)\notag \\
    =& -\frac{1}{2} \biggl( \sum_{m=1}^N D_{+,m}^\varepsilon + \sum_{|k-m|=1} D_{+,k,m}^\varepsilon \biggr)
       -\frac{1}{2} \biggl( \sum_{m=1}^N D_{-,m}^\varepsilon + \sum_{|k-m|=1} D_{-,k,m}^\varepsilon \biggr).
\end{align}

In the following we will show that the positive branch vanishes, that is
\begin{equation}\label{3.69}
\lim_{\varepsilon \to 0+} \left( \sum_{m=1}^ND_{+,m}^\varepsilon + \sum_{|k-m|=1} D_{+,k,m}^\varepsilon \right) = 0.
\end{equation}

\begin{itemize}
    \item \textbf{The case of $|k-m|=0$}
\end{itemize}

Applying integration by parts, we obtain:
\begin{align}\label{3.70}
\sum_{m=1}^ND_{+,m}^\varepsilon
&= \sum_{m=1}^N \int_0^T \int_{\Omega} \theta \xi_m^2 \overline{\mathbf{W}^\varepsilon}_{-,m} \cdot \nabla\!\left( \frac{1}{2} |\overline{\mathbf{W}^\varepsilon}_{+,m}|^2 \right) dx dt \notag\\
&= -\frac{1}{2} \sum_{m=1}^N \int_0^T \int_{V_m} |\overline{\mathbf{W}^\varepsilon}_{+,m}|^2 \xi_m^2 \overline{\mathbf{W}^\varepsilon}_{-,m} \cdot \nabla\theta \, dx dt
   - \frac{1}{2} \sum_{m=1}^N \int_0^T \int_{V_m} \theta |\overline{\mathbf{W}^\varepsilon}_{+,m}|^2 \overline{\mathbf{W}^\varepsilon}_{-,m} \cdot \nabla(\xi_m^2) dx dt \notag\\
&= -\Bigg( \frac{1}{2} \sum_{m=1}^N \int_0^T \int_{V_m} |\overline{\mathbf{W}^\varepsilon}_{+,m}|^2 \xi_m^2 \overline{\mathbf{W}^\varepsilon}_{-,m} \cdot \nabla\theta \, dx dt \notag\\
&\qquad + \frac{1}{2} \sum_{m=1}^N \int_0^T \int_{V_{m-1}\cap V_{m}} \theta |\overline{\mathbf{W}^\varepsilon}_{+,m}|^2 \overline{\mathbf{W}^\varepsilon}_{-,m} \cdot \nabla(\xi_m^2) dx dt \notag\\
&\qquad + \frac{1}{2} \sum_{m=1}^N \int_0^T \int_{V_m\cap V_{m+1}} \theta |\overline{\mathbf{W}^\varepsilon}_{+,m}|^2 \overline{\mathbf{W}^\varepsilon}_{-,m} \cdot \nabla(\xi_m^2) dx dt \Bigg) \notag\\
&= -\Bigg[ \frac{1}{2} \sum_{m=1}^N \int_0^T \int_{V_m} |\overline{\mathbf{W}^\varepsilon}_{+,m}|^2 \xi_m^2 \overline{\mathbf{W}^\varepsilon}_{-,m} \cdot \nabla\theta \, dx dt \notag\\
&\qquad + \frac{1}{2} \sum_{m=1}^{N-1} \int_0^T \int_{V_m\cap V_{m+1}}\theta |\overline{\mathbf{W}^\varepsilon}_{+,m}|^2 \overline{\mathbf{W}^\varepsilon}_{-,m} \cdot \nabla(\xi_m^2 + \xi_{m+1}^2) dx dt \notag\\
&\qquad + \frac{1}{2} \int_0^T \Bigg( \int_{V_0\cap V_1} \theta|\overline{\mathbf{W}}_{+,1}^\varepsilon|^2 \overline{\mathbf{W}}_{-,1}^\varepsilon \cdot \nabla(\xi_1^2) dx dt
      + \int_{V_N\cap V_{N+1}} \theta|\overline{\mathbf{W}}_{+,N}^\varepsilon|^2 \overline{\mathbf{W}}_{-,N}^\varepsilon \cdot \nabla(\xi_N^2) dx \Bigg) dt \Bigg] \notag\\
&= -\Bigg( \frac{1}{2} \sum_{m=1}^N \int_0^T \int_{V_m} |\overline{\mathbf{W}^\varepsilon}_{+,m}|^2 \xi_m^2 \overline{\mathbf{W}^\varepsilon}_{-,m} \cdot \nabla\theta \, dx dt \notag\\
&\qquad - \sum_{m=1}^{N-1} \int_0^T \int_{V_m\cap V_{m+1}} \theta|\overline{\mathbf{W}^\varepsilon}_{+,m}|^2 \overline{\mathbf{W}^\varepsilon}_{-,m} \cdot \nabla(\xi_m\xi_{m+1}) dx dt \Bigg).
\end{align}

\begin{itemize}
    \item \textbf{The case of $|k-m|=1$}
\end{itemize}

By property \eqref{2.25}, we have $\overline{\mathbf u}_m = \overline{\mathbf u}_{m+1}$ and $\overline{\mathbf B}_m = \overline{\mathbf B}_{m+1}$ on the region $V_m \cap V_{m+1}$, which imply $\overline{\mathbf{W}^\varepsilon}_{+,m} = \overline{\mathbf{W}^\varepsilon}_{+,m+1}$ and $\overline{\mathbf{W}^\varepsilon}_{-,m} = \overline{\mathbf{W}^\varepsilon}_{-,m+1}$. Thus we obtain:
\begin{align}\label{3.71}
\sum_{|k-m|=1} D_{+,k,m}^\varepsilon
=& \sum_{|k-m|=1} \int_0^T \int_\Omega \theta \xi_k\xi_m \overline{\mathbf{W}^\varepsilon}_{+,k} \cdot \bigl[ (\overline{\mathbf{W}^\varepsilon}_{-,m} \cdot \nabla) \overline{\mathbf{W}^\varepsilon}_{+,m} \bigr] dx dt \notag\\
=& \sum_{m=1}^{N-1} \int_0^T \int_{V_m\cap V_{m+1}} 2\theta \xi_m\xi_{m+1} \overline{\mathbf{W}^\varepsilon}_{+,m} \cdot \bigl[ (\overline{\mathbf{W}^\varepsilon}_{-,m} \cdot \nabla) \overline{\mathbf{W}^\varepsilon}_{+,m} \bigr] dx dt \notag\\
=& \sum_{m=1}^{N-1} \int_0^T \int_{V_m\cap V_{m+1}} 2\theta \xi_m\xi_{m+1} \overline{\mathbf{W}^\varepsilon}_{-,m} \cdot \nabla\left( \frac{1}{2} |\overline{\mathbf{W}^\varepsilon}_{+,m}|^2 \right) dx dt \notag\\
=& -\Bigg( \sum_{m=1}^{N-1} \int_0^T \int_{V_m\cap V_{m+1}} |\overline{\mathbf{W}^\varepsilon}_{+,m}|^2 \xi_m\xi_{m+1} \overline{\mathbf{W}^\varepsilon}_{-,m} \cdot \nabla\theta \, dx dt \notag\\
&\quad + \sum_{m=1}^{N-1} \int_0^T \int_{V_m\cap V_{m+1}} \theta |\overline{\mathbf{W}^\varepsilon}_{+,m}|^2 \overline{\mathbf{W}^\varepsilon}_{-,m} \cdot \nabla(\xi_m\xi_{m+1}) dx dt \Bigg).
\end{align}

Combining \eqref{3.70}-\eqref{3.71}, 
and focusing on the boundary layer region $\Gamma_{4\delta}$, we obtain:
\begin{align}\label{3.72}
&\left|\sum_{|k-m|=0}D_{+,m}^\varepsilon + \sum_{|k-m|=1} D_{+,k,m}^\varepsilon\right|\notag\\
=& \Bigg| \frac{1}{2} \sum_{m=1}^N \int_0^T \int_{V_m} |\overline{\mathbf{W}^\varepsilon}_{+,m}|^2 \xi_m^2 \overline{\mathbf{W}^\varepsilon}_{-,m} \cdot \nabla\theta \, dx dt\notag \\
& + \sum_{m=1}^{N-1} \int_0^T \int_{V_m\cap V_{m+1}} |\overline{\mathbf{W}^\varepsilon}_{+,m}|^2 \overline{\mathbf{W}^\varepsilon}_{-,m} \cdot \nabla\theta\xi_m\xi_{m+1} dx dt \Bigg| \notag\\
\leq& C \sum_{m=1}^N \int_0^T \int_{\Gamma_{4\delta}} |\overline{\mathbf{W}^\varepsilon}_{+,m}|^2 |\overline{\mathbf{W}^\varepsilon}_{-,m} \cdot \nabla\theta| \, dx dt \notag\\
\leq& C \sum_{m=1}^N \Bigg(\delta  \int_0^T \|\overline{\mathbf{W}^\varepsilon}_{+,m}\|_{L^\infty(\Gamma_{4\delta})}^4 dt \Bigg)^{1/2} \Bigg( \int_0^T \int_{\Gamma_{4\delta}} |\nabla \overline{\mathbf{W}^\varepsilon}_{-,m}|^2 dx dt \Bigg)^{1/2}\notag\\
\leq&C \sum_{m=1}^N \Bigg( \int_0^T \|\overline{\mathbf{W}^\varepsilon}_{+,m}\|_{L^\infty(\Gamma_{4\delta})}^4 dt \Bigg)^{1/2} \Bigg( \delta \int_0^T \int_{\Gamma_{4\delta}} |\nabla \overline{\mathbf{W}^\varepsilon}_{-,m}|^2 dx dt \Bigg)^{1/2},
\end{align}
where the fact that $\text{meas}\left(\Gamma_{4\delta}\right)\le C\delta$ and Hardy inequality \eqref{2.9} are used in the penultimate step.
By the hypothesises \eqref{1.7} and \eqref{1.8}, we could get \eqref{3.72} tends to zero as $\varepsilon \to 0+$.

By a symmetric and identical argument, the negative branch approaches zero similarly. We omit the process and state the conclusion directly:
\begin{equation}\label{3.73}
\lim_{\varepsilon \to 0+} \left( \sum_{|k-m|=0} D_{-,m}^\varepsilon + \sum_{|k-m|=1} D_{-,k,m}^\varepsilon \right)= 0.
\end{equation}

Consequently, we obtain $I_5$ vanishes completely:
\begin{equation}\label{3.75}
\lim_{\varepsilon \to 0+} I_5 = 0.
\end{equation}

Combining all the above steps, we complete the proof of Proposition 3.3.
\end{proof}

\section{Conclusion}

\hspace{1.5em}Recall the resolved energy balance derived in Section 3:
\begin{equation}\label{4.1}
\frac{1}{2}\int_{\Omega}\theta\bigl(|\tilde{\mathbf{u}}^{\delta}(x,T)|^2+|\tilde{\mathbf{B}}^{\delta}(x,T)|^2\bigr)dx
-\frac{1}{2}\int_{\Omega}\theta\bigl(|\tilde{\mathbf{u}}^{\delta}(x,0)|^2+|\tilde{\mathbf{B}}^{\delta}(x,0)|^2\bigr)dx = \sum_{i=1}^{5} R_i .
\end{equation}

Taking the limit \(\varepsilon\to0\) in \eqref{4.1}, the preceding analysis shows that each error term on its right-hand side 
vanishes. The strong convergence of the initial data \((\mathbf{u}_0^\varepsilon,\mathbf{B}_0^\varepsilon)\to(\mathbf{u}_0,\mathbf{B}_0)\) in \(L^2(\Omega)\) implies that the initial resolved energy converges to the initial energy of the ideal system. Hence the total energy is strictly conserved in the inviscid and non-resistive limit. Together with the uniform bounds assumed in the theorem, this energy conservation precludes any anomalous dissipation.

Under the assumptions \eqref{1.5}-\eqref{1.7} and Lemma \ref{l2}, there exists some solutions $(\mathbf u^\varepsilon, \mathbf B^\varepsilon, P^\varepsilon)$ such that, up to some subsequences
\begin{align}
\label{4.2}&\mathbf u^\varepsilon\rightharpoonup\mathbf u\hspace{3pt}\text{in}\hspace{3pt} L^3((0,T);B^{\alpha_1,\infty}_3(\Omega^\delta))\cap L^\infty((0,T);L^2(\Omega)),\\
\label{4.3}&\mathbf B^\varepsilon\rightharpoonup\mathbf B\hspace{3pt}\text{in}\hspace{3pt} L^3((0,T);B^{\alpha_1,\infty}_3(\Omega^\delta))\cap L^\infty((0,T);L^2(\Omega)),\\
\label{4.4}&P^\varepsilon\rightharpoonup P\hspace{3pt}\text{in}\hspace{3pt} L^\frac{3}{2}((0,T);L^\frac{3}{2}(\Omega)).
\end{align}
Then, it is easy to get
\begin{align}
\label{4.5}&\partial_t\mathbf u^\epsilon=-\mu\Delta\mathbf u^\epsilon-\text{div}(\mathbf u\otimes\mathbf u)-\nabla P^\epsilon-\text{div}(\mathbf B\otimes\mathbf B)\in L^\frac{3}{2}((0,T);W^{-1,\frac{3}{2}}(\Omega)),\\
\label{4.6}&\partial_t\mathbf B^\epsilon=\nu\Delta\mathbf B^\epsilon-(\mathbf u^\epsilon\cdot\nabla)\mathbf B^\epsilon+(\mathbf B^\epsilon\cdot\nabla)\mathbf u^\epsilon
\in L^\frac{3}{2}((0,T);W^{-1,\frac{3}{2}}(\Omega)).
\end{align}
Then due to standard Aubin-Lions compactness, we have
\begin{align}
\label{4.7}&\mathbf u^\varepsilon\rightarrow\mathbf u\hspace{3pt}\text{in}\hspace{3pt} L^3((0,T);L^3(\Omega^\delta))\cap C([0,T];L^2_{weak}(\Omega)),\\
\label{4.8}&\mathbf B^\varepsilon\rightarrow\mathbf B\hspace{3pt}\text{in}\hspace{3pt} L^3((0,T);L^3(\Omega^\delta))\cap C([0,T];L^2_{weak}(\Omega)).
\end{align}
When $\varepsilon\rightarrow0+$, it is easy to find
\begin{align}
\label{4.9}\left|\int^T_0\int_\Omega\varepsilon^{p}\nabla\mathbf u^\varepsilon\cdot\nabla\phi\, dxdt\right| &\leq \left(\int^T_0\int_\Omega\varepsilon^{p}|\nabla\mathbf u^\varepsilon|^2\, dxdt\right)^{\frac{1}{2}}\left(\int^T_0\int_\Omega\varepsilon^{p}|\nabla\phi|^2\, dxdt\right)^{\frac{1}{2}}\notag\\
&\leq C\varepsilon^{\frac{p}{2}} \rightarrow 0,
\end{align}
and similarly for the magnetic diffusion term,
\begin{align}
\label{4.10}\left|\int^T_0\int_\Omega\varepsilon^{q}\nabla\mathbf B^\varepsilon\cdot\nabla\phi\, dxdt\right| &\leq \left(\int^T_0\int_\Omega\varepsilon^{q}|\nabla\mathbf B^\varepsilon|^2\, dxdt\right)^{\frac{1}{2}}\left(\int^T_0\int_\Omega\varepsilon^{q}|\nabla\phi|^2\, dxdt\right)^{\frac{1}{2}}\notag\\
&\leq C\varepsilon^{\frac{q}{2}}\rightarrow 0,
\end{align}
which yields, up to a subsequence, the limit \((\mathbf{u},\mathbf{B})\) is the weak solution of the ideal MHD equations.

This completes the proof of Theorem 1.1.

\vskip 0.5cm
\noindent {\bf Acknowledgements}

\vskip 0.1cm
The research of  \v{S}. Ne\v{c}asov\'{a} has been supported by the Praemium Academiae of \v S. Ne\v casov\' a. T. Tang is partially supported by NSFC No. 12371246 and Qing Lan Project of Jiangsu Province.


\begin{thebibliography}{99}


\bibitem{bardos19} C. Bardos, E. S. Titi and E. Wiedemann, Onsager's conjecture with physical boundaries and an application to the vanishing viscosity limit, {\em Commun. Math. Phys.}, 370 (2019), 291--310.

\bibitem{bra} A. Brandenburg and K. Subramanian K, Astrophysical magnetic fields and nonlinear dynamo theory, {\em Phys. Rep.}, 417 (2005), 1-209.



\bibitem{brkic24} P. Brkic and E. Wiedemann, On the vanishing viscosity limit for 3D axisymmetric flows without swirl, {\em Nonlinear Anal. Real World Appl.}, 77 (2024), 104066.


\bibitem{ca} R. Caflisch, I. Klapper, G. Steele, Remarks on singularities, dimension and energy dissipation for
ideal hydrodynamics and MHD, {\em Commun. Math. Phys.}, 184 (1997), 443-455.


\bibitem{chen22} R. M. Chen, Z. Liang and D. Wang, A Kato-type criterion for vanishing viscosity near Onsager's critical regularity, {\em Arch. Ration. Mech. Anal.}, 246 (2022), 535--559.

\bibitem{christian26} S. Christian, E. Wiedemann and J. Wo\'{z}nicki, Strong Convergence of Vorticities in the 2D Viscosity Limit on a Bounded Domain, {\em  J. Nonlinear Sci.}, 36 (2026), Paper No. 22, 23 pp.

\bibitem{cobb} D. Cobb and F. Fanelli, Els\"{a}sser formulation of the ideal MHD and improved lifespan in two space dimensions, {\em J. Math. Pures Appl.}, 169 (2023), 189-236.

\bibitem{const94} P. Constantin, W. E and E. S. Titi, Onsager's conjecture on the energy conservation for solutions of Euler's equation, {\em Commun. Math. Phys.}, 165 (1994), 207--209.

\bibitem{dallas14} V. Dallas and A. Alexakis, The signature of initial conditions on magnetohydrodynamic turbulence, {\em Astrophys. J. Lett.}, 788 (2014), L36 (4pp).


\bibitem{esc17} L. Escauriaza and S. Montaner, Some remarks on the $L^p$ regularity of second derivatives of solutions to non-divergence elliptic equations and the Dini condition, {\em Atti Accad. Naz. Lincei Rend. Lincei Mat. Appl.}, 28 (2017), 49--63.

\bibitem{eyink94} G. L. Eyink, Energy dissipation without viscosity in ideal hydrodynamics I. Fourier analysis and local energy transfer, {\em Physica D}, 78 (1994), 222--240.


\bibitem{faraco20} D. Faraco and S. Lindberg, Proof of Taylor's Conjecture on Magnetic Helicity Conservation, {\em Comm. Math. Phys.}, 373 (2020), 707--738.

\bibitem{faraco24} D. Faraco, S. Lindberg and L. Sz\'ekelyhidi Jr., Magnetic helicity, weak solutions and relaxation of ideal MHD, {\em Comm. Pure Appl. Math.}, 77 (2024), 2387--2412.


\bibitem{feireisl17} E. Feireisl, P. Gwiazda, A. \'Swierczewska-Gwiazda and E. Wiedemann, Regularity and energy conservation for the compressible Euler equations, {\em Arch. Ration. Mech. Anal.}, 223 (2017), 1375--1395.

\bibitem{kang07} E. Kang and J. Lee, Remarks on the magnetic helicity and energy conservation for ideal magneto-hydrodynamics, {\em Nonlinearity}, 20 (2007), 2681-2689.

\bibitem{kato84} T. Kato, Remarks on zero viscosity limit for nonstationary Navier-Stokes flows with boundary, {\em Seminar on nonlinear partial differential equations}, Berkeley, in: Math. Sci. Res. Inst. Publ., vol. 2, 1984, pp. 85-98.

\bibitem{kozono89} H. Kozono, Weak and classical solutions of the two-dimensional magnetohydrodynamic equations, {\em Tohoku Math. J.}, 41 (1989), 471--488.

\bibitem{kufner77} A. Kufner, O. John and S. Fu\v{c}\'ik, {\em Function Spaces}, Noordhoff International Publishing, Leyden and Academia, Prague (1977).

\bibitem{lin15} F. Lin, L. Xu and P. Zhang, Global small solutions of 2-D incompressible MHD system, {\em  J. Differential Equations}, 259 (2015), 5440--5485.


\bibitem{maekawa14} Y. Maekawa, On the inviscid limit problem of the Navier-Stokes equations for half-plane in the analytic setting, {\em Commun. Pure Appl. Math.}, 67 (2014), 1045--1105.

\bibitem{mazz26} A. L. Mazzucato, D. Wang and W. Wei, On the vanishing viscosity limit for incompressible flows with inflow/outflow boundary conditions, {\em  J. Differential Equations}, 467 (2026), 114338, 28pp.

\bibitem{ne} J. Ne\v{c}as, Direct methods in the theory of elliptic equations, Springer Monographs in Mathematics, Springer, Heidelberg, 2012. Translated from the 1967 French original by Gerard Tronel and Alois Kufner.




\bibitem{sam98} M. Sammartino and R. E. Caflisch, Zero viscosity limit for analytic solutions of the Navier-Stokes equation on a half-space. I. Existence for Euler and Prandtl equations, {\em Commun. Math. Phys.}, 192 (1998), 433--461.

\bibitem{ser83} M. Sermange and R. Temam, Some mathematical questions related to the MHD equations, {\em Comm. Pure Appl. Math.}, 36 (1983), 635--664.

\end{thebibliography}
\end{document}